\documentclass[final,onefignum,onetabnum]{siuro210301}

\usepackage{lipsum}
\usepackage{amsfonts}
\usepackage{graphicx}
\usepackage{soul}
\graphicspath{ {/Users/isha_slavin/Desktop/} }
\usepackage{epstopdf}
\usepackage{xcolor}
\usepackage{algorithmic}
\ifpdf
  \DeclareGraphicsExtensions{.eps,.pdf,.png,.jpg}
\else
  \DeclareGraphicsExtensions{.eps}
\fi

\usepackage{enumitem}
\setlist[enumerate]{leftmargin=.5in}
\setlist[itemize]{leftmargin=.5in}

\newsiamremark{remark}{Remark}
\newsiamremark{hypothesis}{Hypothesis}
\Crefname{hypothesis}{Hypothesis}{Hypotheses}
\newsiamthm{claim}{Claim}

\headers{Benchmarking CBO Algorithms}{I. Slavin and D. McKenzie}

\title{Adapting Zeroth Order Algorithms for Comparison-Based Optimization}

\author{Isha Slavin\thanks{New York University
  (\email{ivs225@nyu.edu})}}

\dedication{\small\textit{Project advisor: Daniel McKenzie\thanks{Colorado School of Mines (\email{dmckenzie@mines.edu})}}}

\usepackage{amsopn}

\DeclareMathOperator{\argmin}{argmin}

\ifpdf
\hypersetup{
  pdftitle={CBO Benchmarking},
  pdfauthor={I. Slavin and D. McKenzie}
}
\fi

\begin{document}

\maketitle

\begin{abstract}
Comparison-Based Optimization (CBO) is an optimization paradigm that assumes only very limited access to the objective function $f(x)$. Despite the growing relevance of CBO to real-world applications, this field has received little attention as compared to the adjacent field of Zeroth-Order Optimization (ZOO). In this work we propose a relatively simple method for converting ZOO algorithms to CBO algorithms, thus greatly enlarging the pool of known algorithms for CBO. Via PyCUTEst, we benchmarked these algorithms against a suite of unconstrained problems. We then used hyperparameter tuning to determine optimal values of the parameters of certain algorithms, and utilized visualization tools such as heat maps and line graphs for purposes of interpretation. All our code is available at \url{https://github.com/ishaslavin/Comparison_Based_Optimization}. \\

\end{abstract}

\section{Introduction}
\label{sec:intro}

Zeroth-Order Optimization (ZOO) is a branch of mathematical optimization in which one tries to minimize the objective function:
\begin{displaymath}
  f \colon \mathbb{R}^{n} \to \mathbb{R}.
\end{displaymath}

In this paradigm the gradient $\nabla f(x)$ cannot be accessed, and only function evaluations $f(x)$ are available. Comparison-Based Optimization (CBO), which is the focus of this paper, further restricts what one can calculate. For this form of optimization, we assume the user has very limited access to $f(x)$. More explicitly, it is assumed that the only method to obtain information about the function is to use a \emph{Comparison Oracle}, which---when given $x,y \in \mathbb{R}^n$---returns a single bit of information representing whether $f(x) < f(y)$ or $f(y) < f(x)$. More formally:

\begin{definition}[Comparison Oracle]
  A Comparison Oracle is a function $\mathcal{C}_{f}(\bullet , \bullet) \colon \mathbb{R}^{d} \times \mathbb{R}^{d} \to \{-1, +1\}$ defined as
  \begin{displaymath}
    \mathcal{C}_f(x,y) = \operatorname{sign} \left( f(y)-f(x) \right).
  \end{displaymath}
\end{definition}

In certain applications it is useful to consider an oracle which occasionally returns the incorrect answer. So, we also define a {\em noisy} comparison oracle:

\begin{definition}[Noisy Comparison Oracle]
  A Noisy Comparison Oracle with parameter $p \in [0.5,1]$ is a function $\mathcal{C}_{f}(\bullet , \bullet) \colon \mathbb{R}^{d} \times \mathbb{R}^{d} \to \{-1, +1\}$ defined as
  \begin{displaymath}
    \mathbb{P}\left[\mathcal{C}_f(x,y) = \operatorname{sign}\left( f(y)-f(x) \right)\right] = p.
  \end{displaymath}
  \label{def:Noisy_Oracle}
\end{definition}
There are many reasons to consider a comparison oracle and a variety of natural situations where CBO arises. Consider the case of applying reinforcement learning to real-world tasks, in which a real-valued reward seems erroneous to define \cite{christiano2017deep}. Additionally, in many real-world domains numerical feedback signals are either unavailable or are created arbitrarily to support conventional reinforcement learning algorithms \cite{furnkranz2012preference,wimmer2012generalization}. In these cases, human comparison feedback can be used in comparison oracle form in lieu of a numerical reward function. For example, when attempting to optimize exoskeleton gait researchers determined that since exoskeleton-walking is non-intuitive, users can provide preference between multiple gaits much more reliably than numerically quantifying experience \cite{tucker2020preference}. Another application for using CBO involves optimizing information retrieval systems through maximizing user utility \cite{yue2009interactively}. In this scenario it is infeasible to assign a numerical utility value to a result served to a user, yet simple to obtain judgements of utility by asking a user to compare two results; this is a form of comparison oracle. It is important to note that in such situations, CBO is the only option. Gradients or even function values are simply not available. \\

Despite the wealth of potential applications, relatively few algorithms for CBO have been proposed \cite{cheng2019query,cai2020one,karabag2021smooth}. On the other hand, there is a wealth of algorithms available for ZOO, see for example the recent survey \cite{larson2019derivative}. We observe that some, but not all, ZOO algorithms can be adapted to CBO. The first contribution of this paper is a simple criterion for determining when this is the case, and a procedure for doing so (See \Cref{sec:main}).  \\

There is also a lack of clear comparison between CBO algorithms in prior literature. In particular, it is not clear how such algorithms perform on large scale continuous optimization problems. This makes it difficult for practitioners wishing to use a CBO algorithm in practice to select the appropriate algorithm for their problem. As our second main contribution we provide a software suite for easy benchmarking of CBO algorithms, using the CuTEST set of test functions. We use this to compare five CBO algorithms---two native CBO methods and three that are converted from ZOO methods using the procedure mentioned above. \\

Finally, CBO algorithms often have many hyperparameters, and it is not always clear from theoretical grounds which hyperparameter settings are optimal. Using the software tools mentioned above, we search the hyperparameter space of two CBO algorithms and identify values which empirically work well on the CuTEST problem set. These could be used by practioners working on similar problems.\\

The paper is organized as follows. Our main contributions are available in \Cref{sec:main}, containing a novel utility we introduce. Pseudocode for our modifications of current algorithms are presented in \Cref{sec:main appendix} and experimental results are presented in \Cref{sec:experiments}. A discussion of those results can be found in \Cref{sec:discussion} and \Cref{sec:conclusions}.


\section{From ZOO to CBO}
\label{sec:main}

The motivation for this work was the observation that many ZOO algorithms {\em do not use the function values directly}. Rather, such algorithms proceed by sampling a small number of points $z_1,\ldots, z_m$ near the current iterate $x_k$ and then ranking the function evaluations $f(z_1),\ldots, f(z_m)$. The next iterate is then determined using this ranking. We note that such a ranking can be done using only a comparison oracle, and formalize this as \Cref{prop:Comparisons Only}.

\begin{property}[The ``Comparison Only'' property]
\label{prop:Comparisons Only}
Suppose $\mathcal{A}$ is a ZOO algorithm. We say $\mathcal{A}$ satisfies the ``Comparison Only'' property if it only uses function values within an $\argmin$ over a finite set, {\em e.g} $\argmin_{i=1,\ldots,m} f(z_i)$ or to sort a list $z_1,\ldots, z_m$ according to $f(z_i)$ {\em i.e.} to find a permutation $\pi$ such that $f(z_{\pi(i)}) \leq f(z_{\pi(2)}) \leq \ldots \leq f(z_{\pi(m)})$.
\end{property}

If $\mathcal{A}$ satisfies \Cref{prop:Comparisons Only}, it can easily be converted to a CBO algorithm using the utility \Cref{alg:CompSort} (an implementation of {\tt Bubble Sort} \cite[Section 2.3]{cormen2022introduction} using the comparison oracle) or \Cref{alg:CompMax} (an implementation of the {\tt Minimum} algorithm \cite[Section 9.1]{cormen2022introduction} using the comparison oracle). Observe that the above modification only holds when the zeroth order algorithm finds an $\argmin$ over a finite set; otherwise, \Cref{prop:Comparisons Only} will not hold.  We illustrate this with the Stochastic Three Point ({\tt STP}) method \cite{bergou2020stochastic}, see \Cref{alg:stp original_2}. For an example of a ZOO algorithm which does not satisfy our condition, and so cannot be converted into a CBO algorithm, consider the {\tt RSGF} algorithm of \cite{ghadimi2013stochastic} (see also \cite{nesterov2017random}). Here, at each iteration a gradient estimator is constructed from function evaluations and used in place of the gradient:

\begin{align}
    & \hat{g}_k  = \frac{f(x_k + \delta u_i) - f(x_k)}{\delta} \approx \nabla f(x_k) \\
    & x_{k+1} = x_k - \alpha\hat{g}_k
\end{align}
A second example could be any interpolation-based method, {\em e.g.} {\tt NEWUOA} \cite{powell2006newuoa}, or the recently introduced {\tt HJ-Mad} \cite{heaton2022global}. That {\tt CMA-ES} (and related algorithms) satisfy the Comparison Only property appears to be well-known in the evolutionary computing community, where it is frequently mentioned as a source of robustness \cite{gelly2007comparison}.

\begin{algorithm}
\caption{Comparison-based Sort ({\tt CompSort})}
\label{alg:CompSort}
\begin{algorithmic}
\STATE{\textbf{Initialization}}
\STATE{Take in Comparison Oracle$\colon \mathcal{C}_{f}$, lst $= [z_1,z_2,...,z_m]\colon$list of input values}
\FOR{$i = 1,\ldots,m-1$}
    \FOR{$j = 1,\ldots,m-i-2$}
        \IF{$\mathcal{C}_{f}\left(\text{lst}[j+1], \text{lst}[j]\right) = +1$} 
            \STATE{Swap lst[$j$] and lst[$j+1$]}
        \ENDIF
    \ENDFOR
\ENDFOR
\RETURN Sorted list $[z_{\pi(1)}, z_{\pi(2)}, \ldots, z_{\pi(m)}]$.
\end{algorithmic}
\end{algorithm}

\begin{algorithm}
\caption{Comparison-based Min ({\tt CompMin})}
\label{alg:CompMax}
\begin{algorithmic}
\STATE{\textbf{Initialization}}
\STATE{Take in Comparison Oracle$\colon \mathcal{C}_{f}$, lst $= [z_1,z_2,...,z_m]\colon$list of input values}
\STATE $z_{+} = z_1$
\FOR{$k=2,\ldots,m$}
    \IF{$\mathcal{C}_f(z_{+},z_k) = +1$}
        \STATE{$z_{+} = z_{+}$}  
    \ELSIF{$\mathcal{C}_f(z_{+},z_k) = -1$}
        \STATE{$z_+ = z_k$}
    \ELSE
        \STATE{$z_{+} = z_{+}$ or $z_{+} = z_{k}$ with equal probability.}
    \ENDIF
\ENDFOR
\RETURN $z_+$

\end{algorithmic}
\end{algorithm}

The two procedures introduced above are used to convert ZOO to CBO algorithms. Instead of using direct function evaluations to sort a list (\Cref{alg:CompSort}) or find the argmin (\Cref{alg:CompMax}), it queries a Comparison Oracle. It then uses the one-bit comparisons outputted by the oracle to either return a list of input vectors, arranged by function values in ascending order (\Cref{alg:CompSort}) or output the input which yields the smallest function value (\Cref{alg:CompMax}). This is done without ever evaluating the objective function at an input directly. \\

As proof of concept, we identified three ZOO algorithms satisfying \Cref{prop:Comparisons Only} and transform them to CBO algorithms using \Cref{alg:CompMax} or \Cref{alg:CompSort}. The algorithms we consider are  the Stochastic Three Points Method ({\tt STP}) \cite{bergou2020stochastic}, Covariance Matrix Adaptation Evolutionary Strategies ({\tt CMA-ES}) \cite{hansen2016cma}, and Gradientless Descent ({\tt GLD}) \cite{golovin2019gradientless}. For all algorithms considered, we have the functionality to generate multiple types of distributions $\mathcal{D}$ including the Uniform distribution over $\{e_1,\ldots, e_n\}$, where $e_i$ denotes the $i$-th canonical basis vector, Gaussian, Uniform over the unit sphere, and Rademacher (see \Cref{sec:main appendix}). Note that direct search methods such as the Nelder-Mead simplex algorithm \cite{Nelder:1965}, Powell's method \cite{1964Powell}, and various directional direct search algorithms \cite{larson2019derivative} also satisfy \Cref{prop:Comparisons Only} and can thus be converted to Comparison-Based algorithms. \\

\Cref{alg:stp original_2} shows how we were able to convert the STP optimization algorithm from Zeroth-Order to Comparison-Based. When using step $3a$ the algorithm uses function evaluations to determine the argmin over a set of three input vectors. When using a modification, step $3b$, \Cref{alg:CompMax} is employed to find the argmin, thus side-stepping function evaluations and using a Comparison Oracle instead. Note that the four distributions mentioned above can be used to randomly generate random vectors $s_k$ in the algorithm below. For CBO conversion of {\tt CMA-ES} and {\tt GLD} zeroth-order algorithms see \Cref{sec:main appendix}.

\begin{algorithm}[H]
\caption{Stochastic Three Point ({\tt STP}). For original algorithm use 3a. For comparison-based version, use 3b.}
\label{alg:stp original_2}
\begin{algorithmic}
\STATE{\textbf{Initialization}}
\STATE{Choose $x_0 \in \mathbb{R}^{n}$, stepsizes $\alpha_k > 0$, probability distribution $\mathcal{D}$ on $\mathbb{R}^{n}$}
\FOR{$k = 0,1,2,....$}
  \STATE{1. Generate a random vector $s_k \sim \mathcal{D}$}
  \STATE{2. Let $x_+ = x_k + \alpha_k s_k$ and $x_- = x_k - \alpha_k s_k$}
  \STATE{3a. $x_{k+1} = \argmin\{f(x_-), f(x_+), f(x_k)\}$}
  \STATE{3b. $x_{k+1} = {\tt CompMin}(x_{-}, x_{+}, x_k)$}
\ENDFOR
\end{algorithmic}
\end{algorithm}

Converting a ZOO algorithm to a CBO one using our utilities changes the {\em query complexity} of the algorithm, defined as the number of function evaluations (resp. comparison oracle queries) required to find a suitable solution, in a predictable way:

\begin{theorem}
Suppose $\mathcal{A}$ is a ZOO algorithm satisfying \Cref{prop:Comparisons Only} and making $m$ function evaluations per iteration. Then the associated CBO algorithm constructed using \Cref{alg:CompSort} (resp. \Cref{alg:CompMax}) makes at most $m^2$ (resp $m-1$) oracle queries per iteration.
\end{theorem}
\begin{proof}
This follows from standard complexity analysis of Bubble Sort and Minimization, see \cite{cormen2022introduction}.
\end{proof}

\begin{remark}
When $m$ is large,  using \Cref{alg:CompSort} increases the query complexity of an algorithm significantly, as it requires $m^2$ queries per iteration. This could be improved by using a more sophisticated sorting algorithm, {\em e.g.} QuickSort. Nonetheless, care must be taken when adapting ZOO algorithms requiring many sort operations. In particular, we caution against naive comparison-based implementations of the Nelder-Mead simplex algorithm \cite{nelder1965simplex}, as this may make as many as $\mathcal{O}(n^2)$ queries per iteration. 
\end{remark}
\section{A benchmarking utility for CBO algorithms}
\label{sec:Benchmark-Utility}
To the best of our knowledge, there does not exist a suite of test problems for benchmarking CBO algorithms. So, we create one using the well-known CuTEST package \cite{gould2015cutest} and the PyCUTEst interface \cite{fowkes2019PyCUTEst}. We do so by providing a simple wrapper which turns any test function $f$ into a comparison oracle $\mathcal{C}_f$; see \Cref{alg:comparison oracle utility}. We also provide a wrapper allowing for noisy comparison oracles; see \Cref{alg:noisy oracle}. These functions are available at \url{https://github.com/ishaslavin/Comparison_Based_Optimization}.  

\begin{algorithm}
\caption{Oracle Utility $(\mathcal{C}_f)$}
\label{alg:comparison oracle utility}
\begin{algorithmic}
\STATE{\textbf{Initialization}}
\STATE{Take in function $f \colon \mathbb{R}^{n} \to \mathbb{R}$, $x \colon$ first input value, $y \colon$ second input value}
\IF{$f(x) < f(y)$}
  \RETURN 1
\ELSIF{$f(y) < f(x)$}
  \RETURN -1
\ELSE
  \RETURN 0
\ENDIF
\end{algorithmic}
\end{algorithm}

\begin{algorithm}
\caption{Noisy Oracle Utility}
\label{alg:noisy oracle}
\begin{algorithmic}
\STATE{\textbf{Initialization}}
\STATE{Take in function $f \colon \mathbb{R}^{n} \to \mathbb{R}$, $p \in  [0, 1] \colon$ noisy-ness of oracle, $x \colon$ first input value, $y \colon$ second input value}
\STATE{Generate $r \in [0, 1]$ randomly}
\IF {$r < p$}
  \IF{$f(x) < f(y)$}
    \RETURN 1
  \ELSIF{$f(y) < f(x)$}
    \RETURN -1
  \ELSE
    \RETURN 0
  \ENDIF
\ELSIF{$r \geq p$}
  \IF{$f(x) < f(y)$}
    \RETURN -1
  \ELSIF{$f(y) < f(x)$}
    \RETURN 1
  \ELSE
    \RETURN 0
  \ENDIF
\ENDIF
\end{algorithmic}
\end{algorithm}

\section{Experimental results}
\label{sec:experiments}
We empirically compare five CBO algorithms. Two are specifically designed for CBO problems ({\tt SCOBO} \cite{cai2020one} and {\tt SignOPT} \cite{cheng2019query}) while three ({\tt GLD}, {\tt CMA-ES}, and {\tt STP}) are adapted from ZOO algorithms using the method of \Cref{sec:main}. See \Cref{sec:main} for further details, and \Cref{sec:main appendix} for pseudocode. We first benchmark these algorithms on three simple test problems: Sparse Quadratic, MaxK, and (non-sparse) Quadratic, studied in \cite{cai2020one,cai2021zeroth}.

\begin{definition}[SparseQuadratic]
For fixed parameters $n=200$ and $k=20$, define
\begin{align}
    & f: \mathbb{R}^{n} \to \mathbb{R} \\
    & f(x) = \sum_{i=1}^k x_i^2
\end{align}
\label{def:SparseQuadratic}
\end{definition}

\begin{definition}[MaxK]
  Fix the parameters $n=200$ and $k=20$. For any $x \in \mathbb{R}^n$, let $\pi$ denote a permutation such that $x_{\pi(1)} \geq x_{\pi(2)} \geq \dots \geq .$ Define:
  \begin{align}
    & f_{\mathrm{max-k}} \colon \mathbb{R}^{n} \to \mathbb{R} \\
    & f_{\mathrm{max-k}}(x) = \sum_{i=1}^{k}{x^2}_{\pi(i)}
  \end{align}

\end{definition}

By non-sparse quadratic we mean the function
\begin{definition}[NonSparseQuadratic]
\begin{align}
    & f: \mathbb{R}^{200} \to \mathbb{R} \\
    & f(x) = \sum_{i=1}^{200}x_i^2
\end{align}
\label{def:NonSparseQuadratic}
\end{definition}

\begin{figure}[H]
  \centering
  \label{fig:ToyProblemBenchmarking}
  \includegraphics[width=0.32\textwidth]{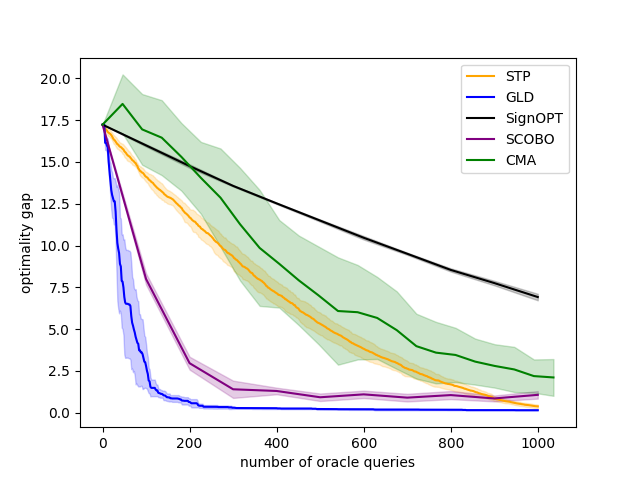}
  \hfill
  \includegraphics[width=0.32\textwidth]{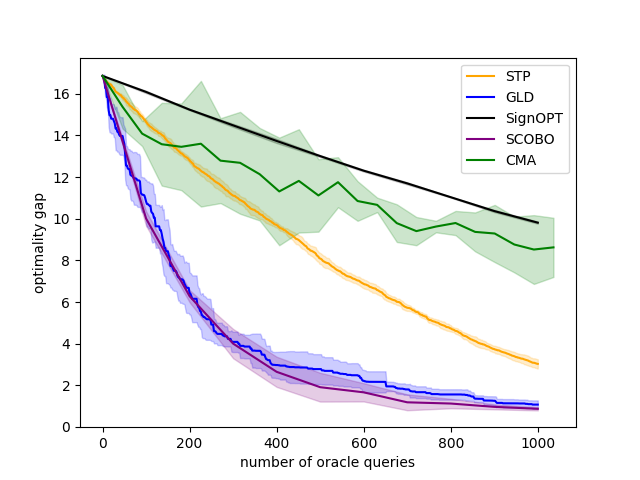}
  \hfill
  \includegraphics[width=0.32\textwidth]{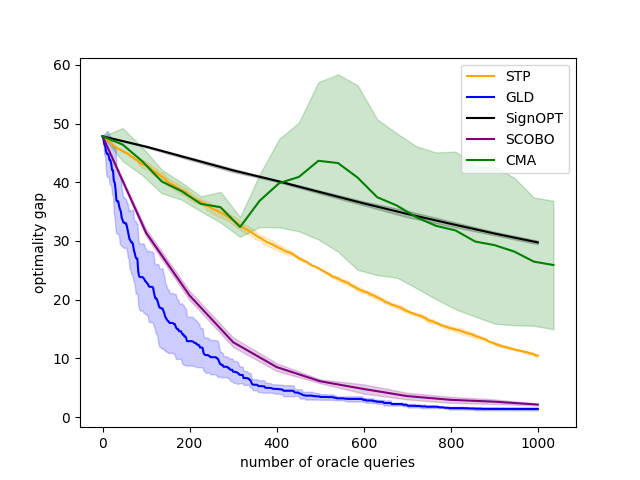}
  \caption{{\bf Left:} SparseQuadratic. {\bf Center:} MaxK. {\bf Right:} NonSparseQuadratic. Graphs display the mean optimality gap plotted against the cumulative number of comparison oracle queries for CBO algorithms against the three functions mentioned above. }
\end{figure}

We test our five algorithms on these three functions. For each problem, we run each algorithm using the same initial point $x_0$ and repeat this five times. \Cref{fig:ToyProblemBenchmarking} shows the mean optimality gap ({\em i.e.} $f(x_k) - f(x_{\star})$) plotted against the cumulative number of comparison oracle queries. The shading indicates the min--max range. As expected \cite{cai2020one} {\tt SCOBO} optimizes the fastest on the function MaxK which exhibits gradient sparsity, {\em i.e.}
\begin{equation}
    \left\|\nabla f(x)\right\|_{0} := \left|\{i: \nabla_if(x) \neq 0\}\right| \leq 20 \text{ for all } x \in \mathbb{R}^n.
\end{equation}
For the function without gradient sparsity (the non-sparse quadratic) as well as SparseQuadratic, {\tt GLD} (originally zeroth-order but modified here to be comparison-based) performs best. {\tt STP}, {\tt SignOPT}, and {\tt CMA} show similar patterns of minimizing the functions linearly and not fast.

\subsection{PyCUTEst Results}
The functions mentioned above are fairly simple. To benchmark against problems that generalize to an overall population of functions we utilized the PyCUTEst  \cite{fowkes2019PyCUTEst} Python wrapper to the Fortran package CUTEst \cite{gould2015cutest}, used to test optimization software. Available in this package are 117 unconstrained problems, each varying in input vector dimension. We benchmarked our CBO algorithms against 22 of these problems, ranging in input vector dimension from $\mathbb{R}^{10}$ to $\mathbb{R}^{100}$. The 22 PyCUTEst functions used and their dimensions are provided in \Cref{tab:PyCUTEst_problems_dimensions}. To turn these functions into CBO problems we used the utility described in \Cref{sec:Benchmark-Utility}. Representative results are shown in \Cref{fig:PyCUTEst_Results}. To gain a clearer perspective on how these CBO algorithms compare over the entire benchmark set, we use performance profiles \cite{dolan2002benchmarking}. As described in \cite{kim2021curvature}, performance profiles are constructed as follows: \\

Let $\mathcal{P}$ denote the set of benchmark problems and $\mathcal{S}$ denote the set of algorithms under consideration. For each $p \in \mathcal{P}$ and $s \in \mathcal{S}$ the {\em performance ratio} $r_{p,s}$ is defined by 
\begin{equation*}
    r_{p,s} = \frac{t_{p,s}}{\min_{s'\in\mathcal{S}} t_{p,s'}},
\end{equation*}
where $t_{p,s}$ is the number of comparison oracle queries required for algorithm $s$ to solve problem $p$ (lower is better). So, $r_{p,s}$ represents the performance of $s$ on $p$ relative to the best algorithm in $\mathcal{S}$ for $p$. The {\em performance profile} of $s$, $\rho_{s} : [1,\infty) \rightarrow [0,1]$ is
\begin{align*}
    \rho_s(\tau) = \frac{|\{p \in \mathcal{P} : r_{p,s} \leq \tau \}|}{|\mathcal{P}|}.
\end{align*}
In other words, $\rho_s(1)$ is the fraction of problems for which $s$ solves the problem first, so a higher value of $\rho_s(1)$ is better. When $\tau$ is larger $\rho_s(\tau)$ represents the fraction of problems for which the performance of algorithm $s$ is at most $\tau$ times worse than the performance of the best algorithm tested on this problem. Again, higher values of $\rho_s(\tau)$ are preferred and indicate that the algorithm $s$ is {\em robust}. Our success condition for performance profiling is determined by the {\em relative} size of either the function values or the gradient. More specifically, we define two success criterion terms to profile against: one in which the final function evaluation $f(x_k)$ is 0.05 times the initial function evaluation $f(x_0)$, and one in which the euclidean norm of the gradient of the function evaluated at $x_k$ is 0.05 times the 2-norm of the gradient of the function evaluated at the starting input $x_0$. Performance profiles for {\tt SignOPT}, {\tt GLD}, {\tt CMA-ES}, {\tt SCOBO}, and {\tt STP}, tested on the 22 problems in \Cref{tab:PyCUTEst_problems_dimensions}, are shown in \Cref{fig:Performance_Profiles}. \\

\begin{table}[H]
\footnotesize
\caption{PyCUTEst problems used in benchmarking.}
\label{tab:PyCUTEst_problems_dimensions}
\begin{center}
  \begin{tabular}{|c|c|c|c|c|} \hline
   Problem & \bf CHNROSNB & \bf CHNRSNBM & \bf ERRINROS & \bf ERRINRSM \\ \hline
   Dimension & 50 & 50 & 50 & 50 \\ \hline
   Problem & \bf HILBERTB & \bf QING & \bf LUKSAN11LS & \bf LUKSAN12LS \\ \hline
   Dimension & 10 & 100 & 100 & 98 \\ \hline
   Problem & \bf LUKSAN13LS & \bf LUKSAN14LS & \bf LUKSAN15LS & \bf LUKSAN16LS \\ \hline
   Dimension & 98 & 98 & 100 & 100 \\ \hline
   Problem & \bf LUKSAN17LS & \bf LUKSAN21LS & \bf LUKSAN22LS & \bf MANCINO \\ \hline
   Dimension & 100 & 100 & 100 & 100 \\ \hline
   Problem & \bf STRTCHDV & \bf SENSORS & \bf VANDANMSLS & \bf WATSON \\ \hline
   Dimension & 10 & 100 & 22 & 12 \\ \hline
   Problem & \bf TRIGON1 & \bf TRIGON2 & & \\ \hline
   Dimension & 10 & 10 & & \\ \hline
  \end{tabular}
\end{center}
\end{table}

\begin{figure}
  \centering
  \label{fig:PyCUTEst_Results}
  \includegraphics[width=0.32\textwidth]{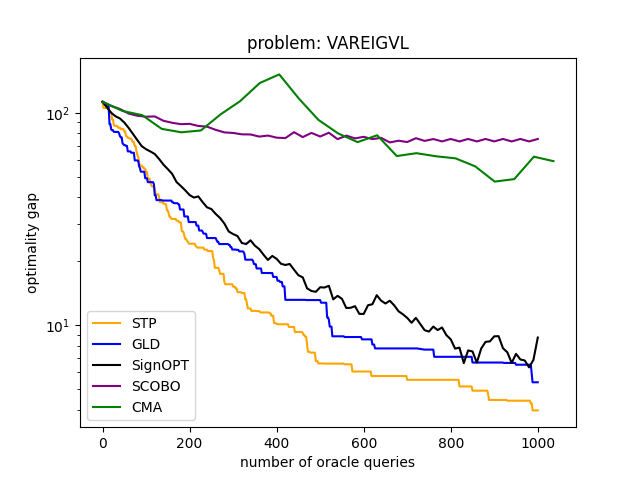}
  \hfill
  \includegraphics[width=0.32\textwidth]{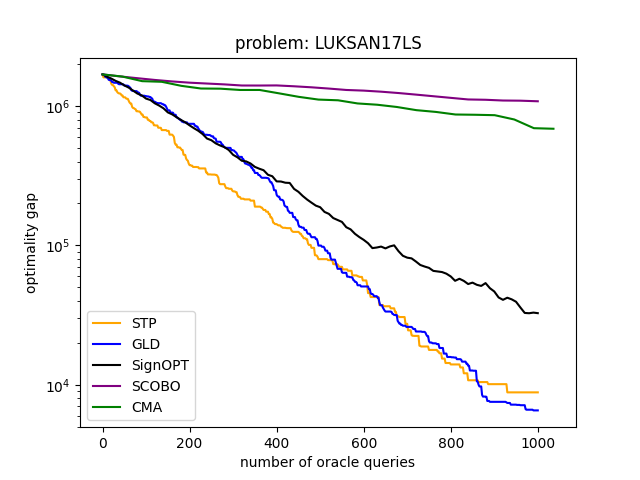}
  \hfill
  \includegraphics[width=0.32\textwidth]{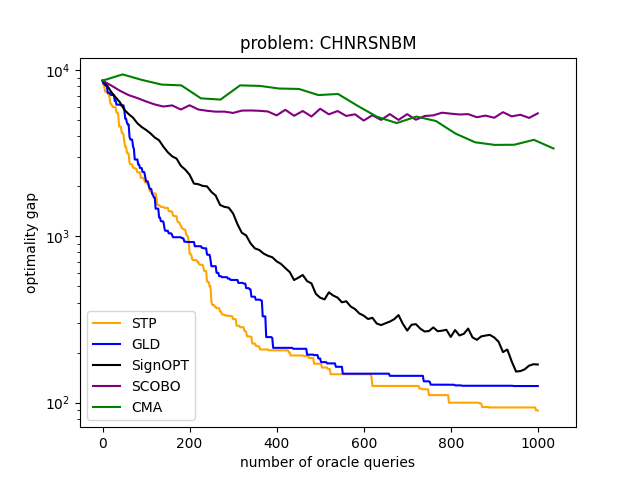}
  \caption{Optimality gap {\em ($f(x_k) - f(x_{\star})$)} {\em vs} number of iterations for three typical PyCUTEst functions. {\bf Left:} VAREIGVL. {\bf Center:} LUKSAN17LS. {\bf Right:} CHNRSNBM.}
\end{figure}

\begin{figure}
    \centering
    \includegraphics[width=0.45\textwidth]{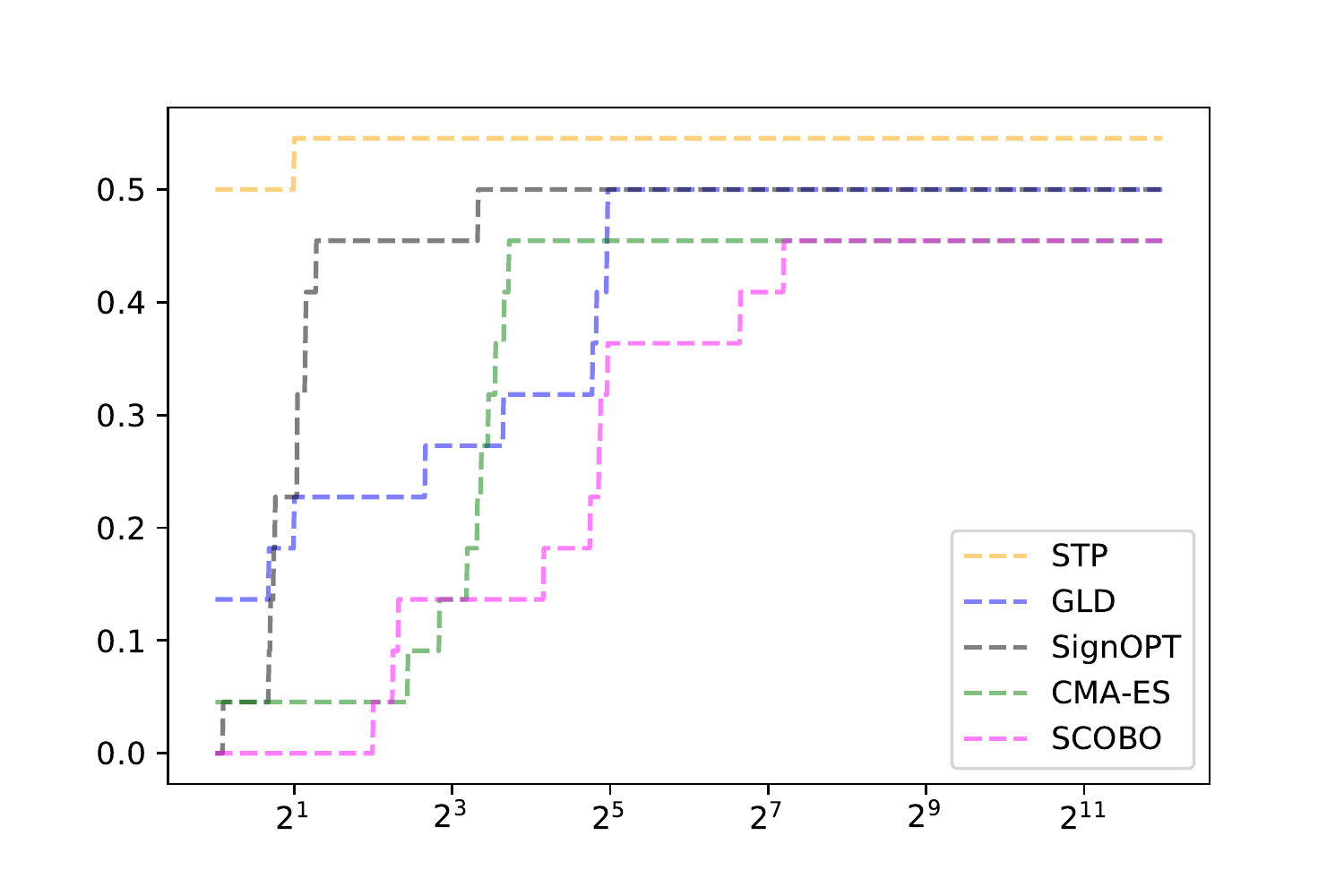}\hfill
     \includegraphics[width=0.45\textwidth]{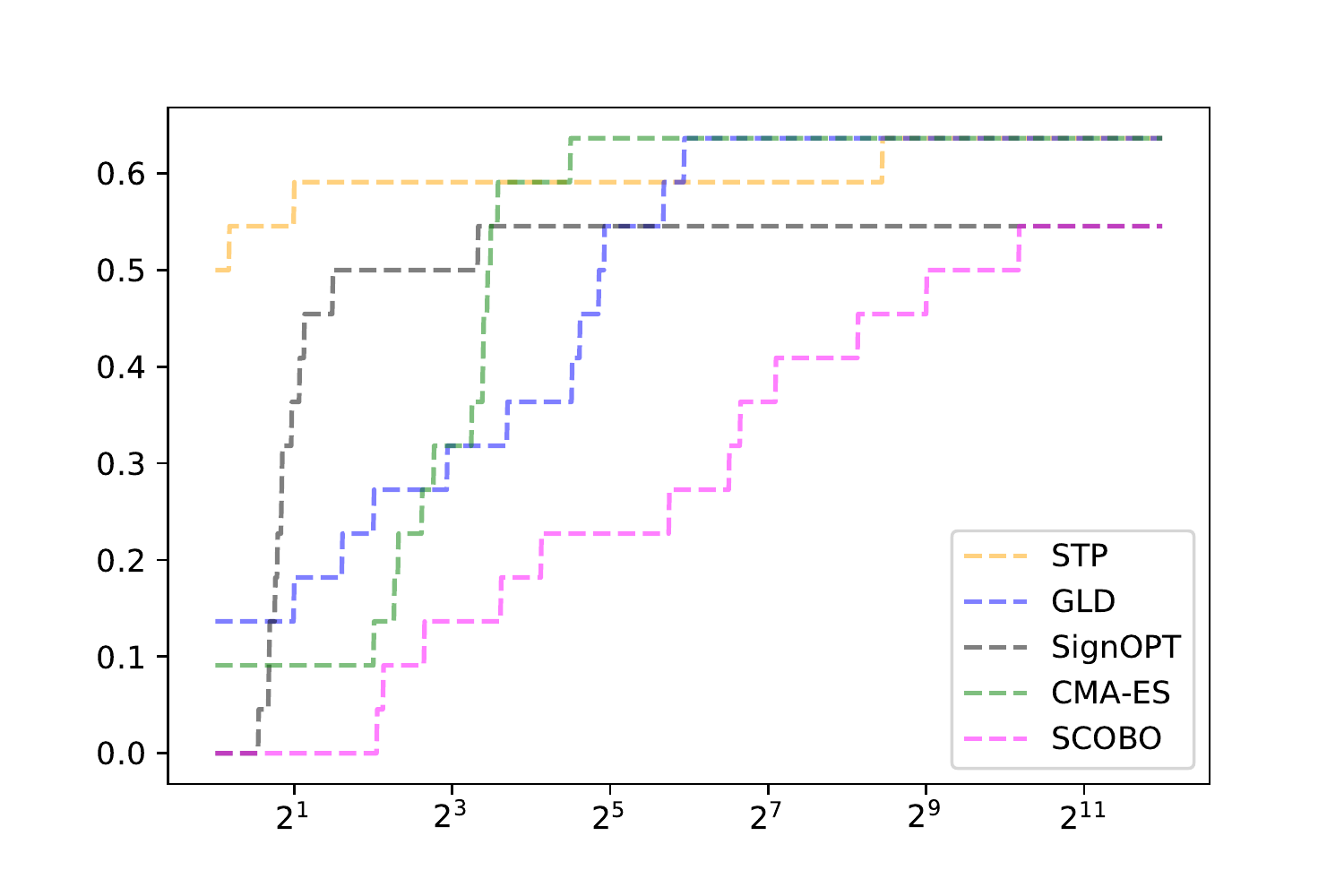} \\
     \includegraphics[width=0.45\textwidth]{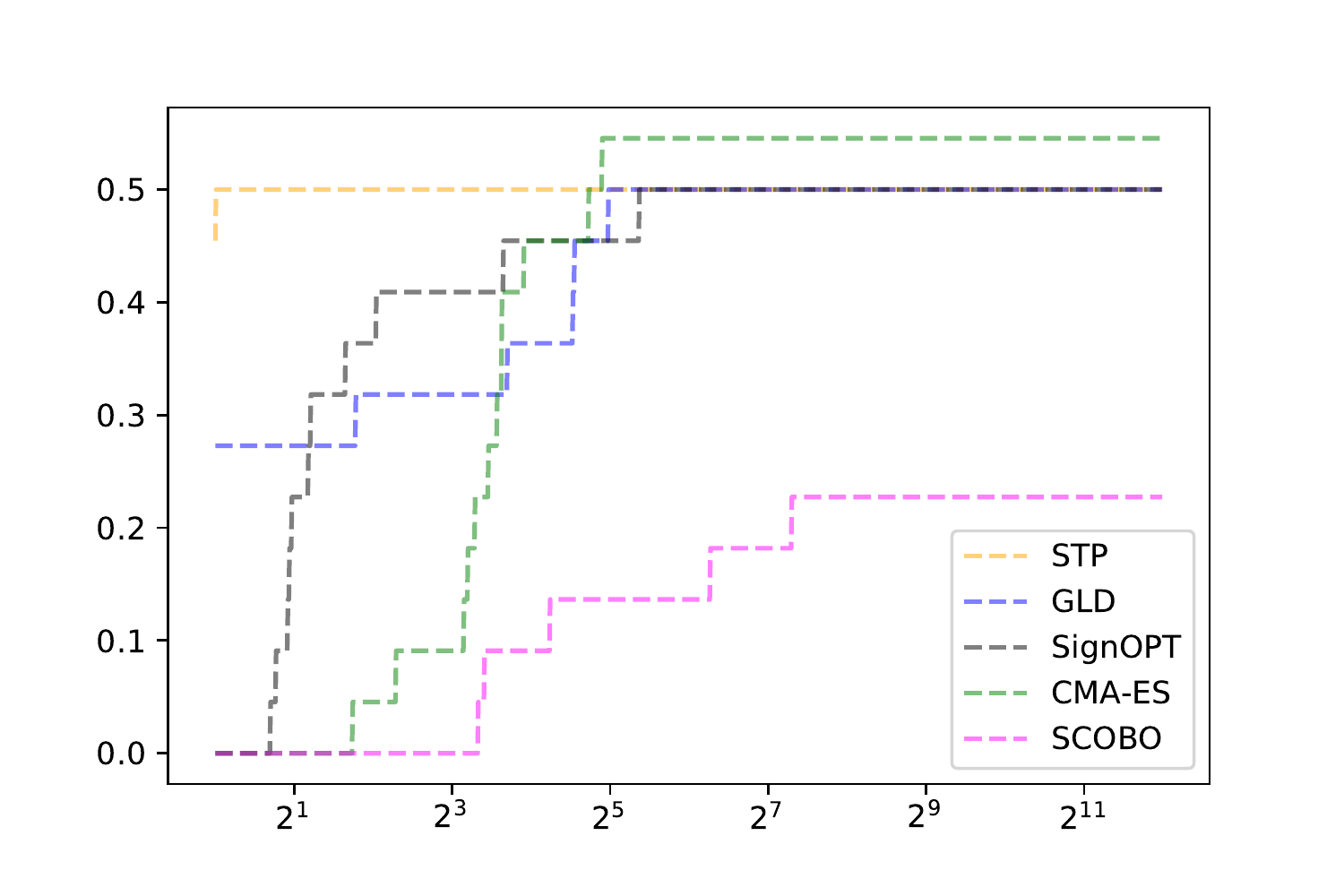}\hfill
     \includegraphics[width=0.45\textwidth]{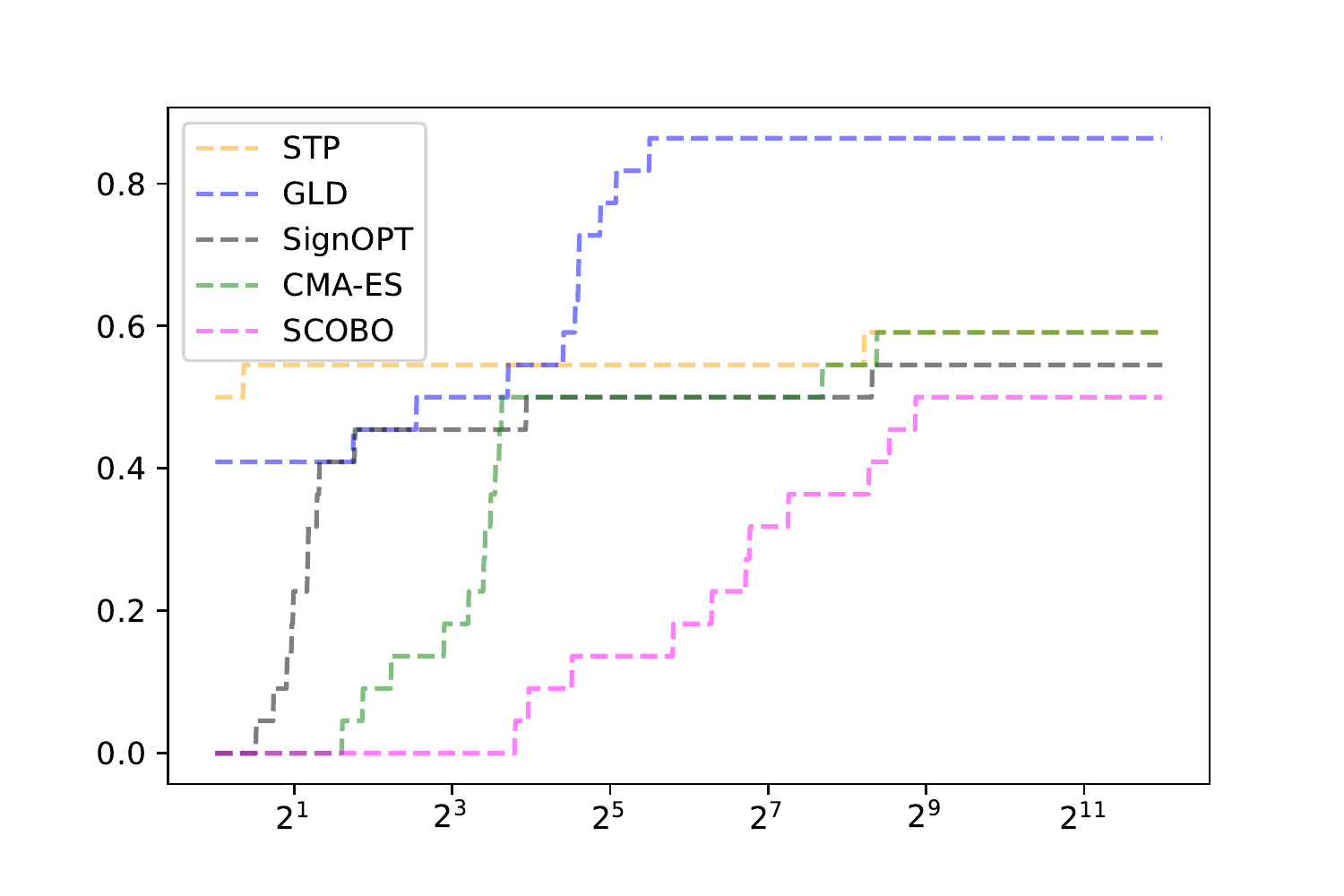} \\
    \caption{Performance profiles for {\tt GLD}, {\tt SignOPT}, {\tt STP}, {\tt CMA-ES}, and {\tt SCOBO}. {\bf Top Left:} A query budget of $10^4$ and success criterion $f(x_k)\leq 0.05f(x_0)$. {\bf Top Right:} A query budget of $10^5$ and success criterion $f(x_k)\leq 0.05f(x_0)$. {\bf Bottom Left:} A query budget of $10^4$ and success criterion $\|\nabla f(x_k)\|_2 \leq 0.05\|\nabla f(x_0)\|_2$. {\bf Bottom Right:} A query budget of $10^5$ and success criterion $\|\nabla f(x_k)\|_2 \leq 0.05\|\nabla f(x_0)\|_2$.}
    \label{fig:Performance_Profiles}
\end{figure}

\subsection{Noisy Oracle} 
\label{sec:noisy oracle subsection}

In practice comparison oracles may occasionally be unreliable. To test the robustness of the CBO algorithms to noise, we ran experiments using the SparseQuadratic function (see \Cref{def:SparseQuadratic}) with the noisy comparison oracle utility (see \Cref{def:Noisy_Oracle}). The noisy oracle takes in a parameter $p$ determining the probability that its output is accurate. We ran experiments using $p=0.7$ and $p=0.9$. The results are shown in \Cref{fig:Noisy_Oracle_Results}. Notice that {\tt SCOBO} and {\tt GLD} show similar trends, while {\tt STP}, {\tt SignOPT}, and {\tt CMA} show similarities as well with a greater margin of error. Additionally, we can see that {\tt CMA} and {\tt GLD} show low robustness to noise with greater margins of error and a clear struggle to minimize the function, whereas {\tt SCOBO} maintains performance in the presence of noise well.

\begin{figure}
  \centering
  \label{fig:Noisy_Oracle_Results}
  \includegraphics[width=0.49\textwidth]{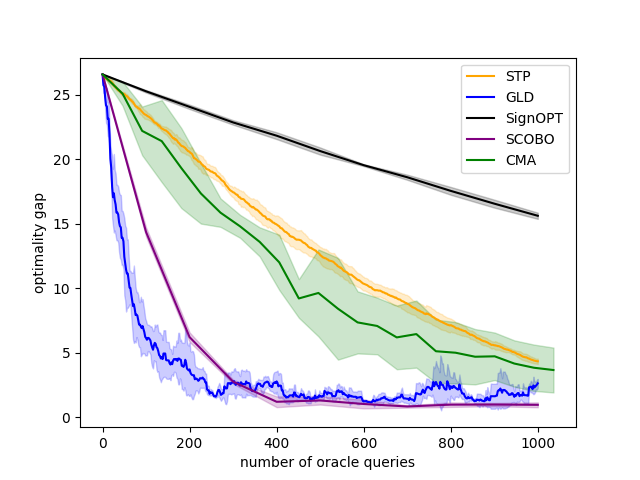}
  \hfill
  \includegraphics[width=0.49\textwidth]{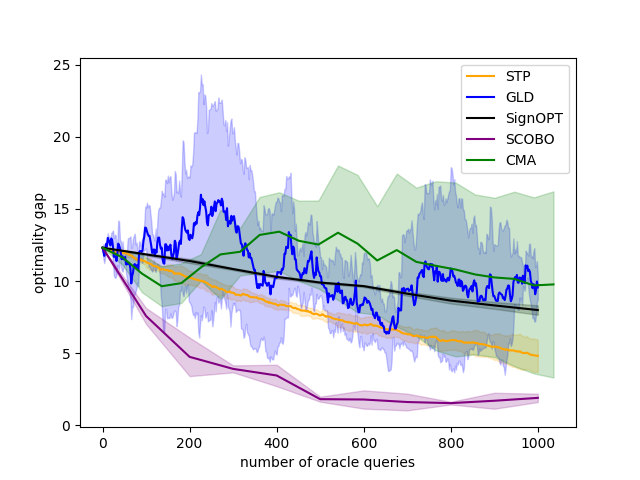}
  \caption{Minimizing the SparseQuadratic function using a noisy comparison oracle. {\bf Left:} $p=0.9$ {\bf Right:} $p=0.7$. }
\end{figure}

\subsection{Hyperparameter Tuning} 
\label{sec:hyperparameter tuning subsection}

To demonstrate how our PyCUTEst utility might be used in practice, we tune the  hyperparameters for two algorithms, {\tt GLD} and {\tt SCOBO}, using three PyCUTEst functions: Rosenbrock (ROSENBR), Hilbert (HILBERTA), and Watson (WATSON). For more information on these functions, see \cite{fowkes2019PyCUTEst}. \\

\begin{figure}
  \centering
  \label{fig:heatmap_GLD}
  \includegraphics[width=0.32\textwidth]{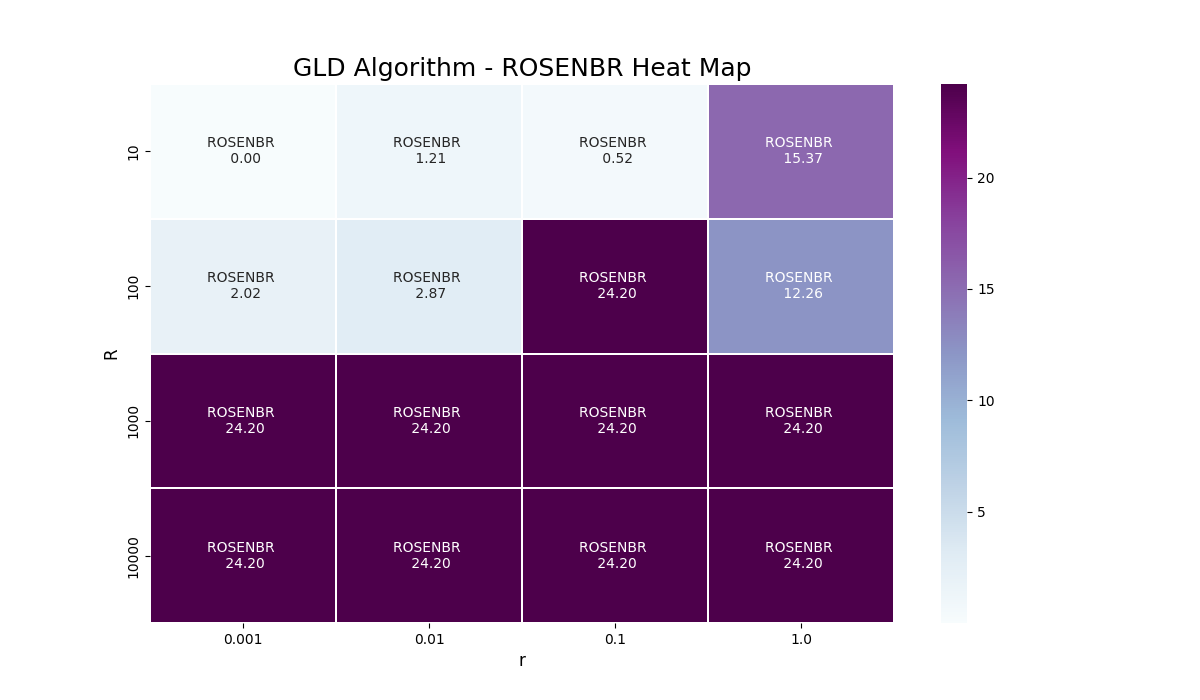}
  \hfill
  \includegraphics[width=0.32\textwidth]{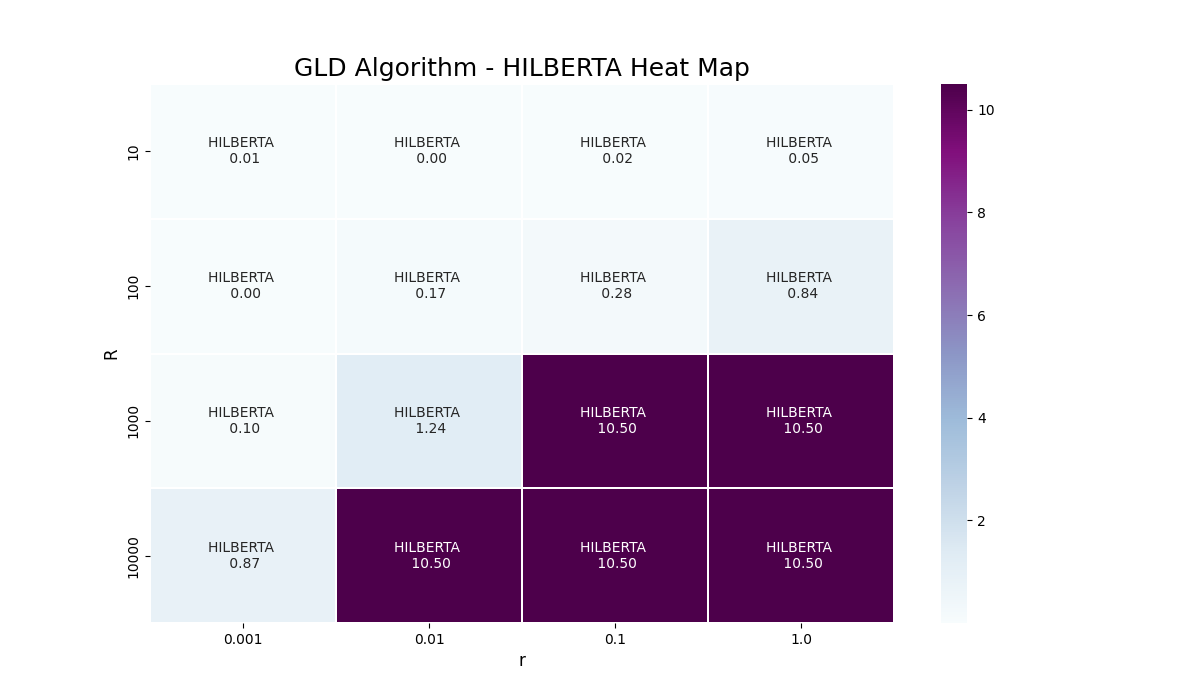}
  \hfill
  \includegraphics[width=0.32\textwidth]{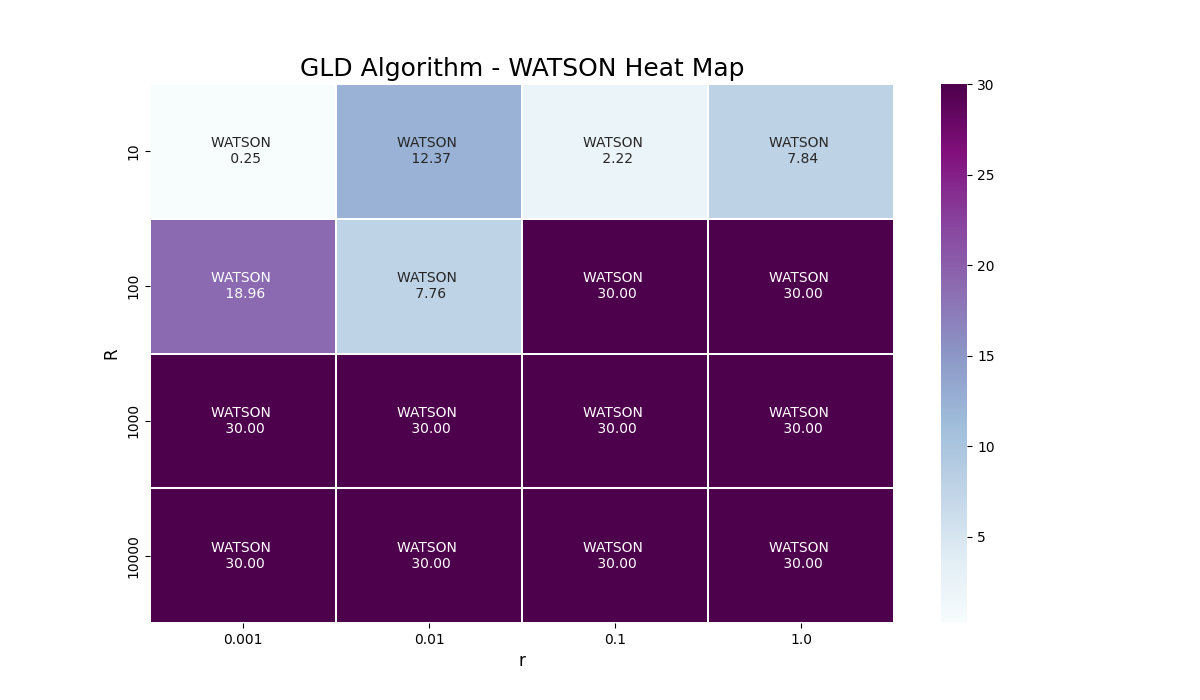}
  \hfill
  \includegraphics[width=0.32\textwidth]{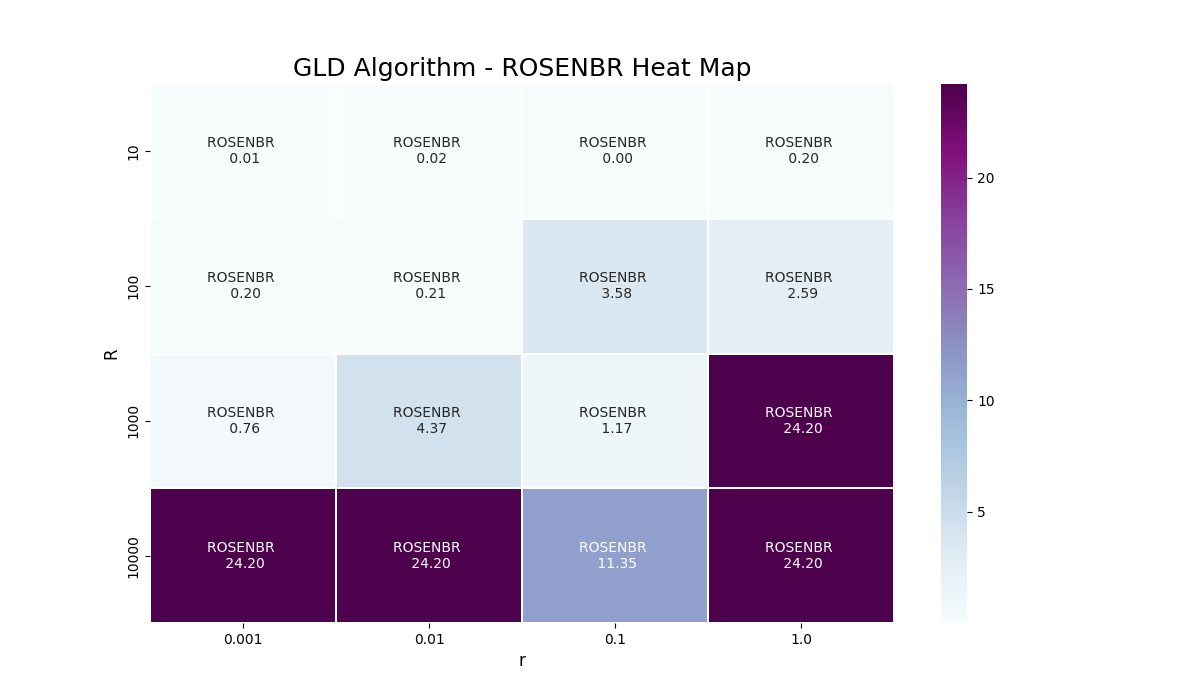}
  \hfill
  \includegraphics[width=0.32\textwidth]{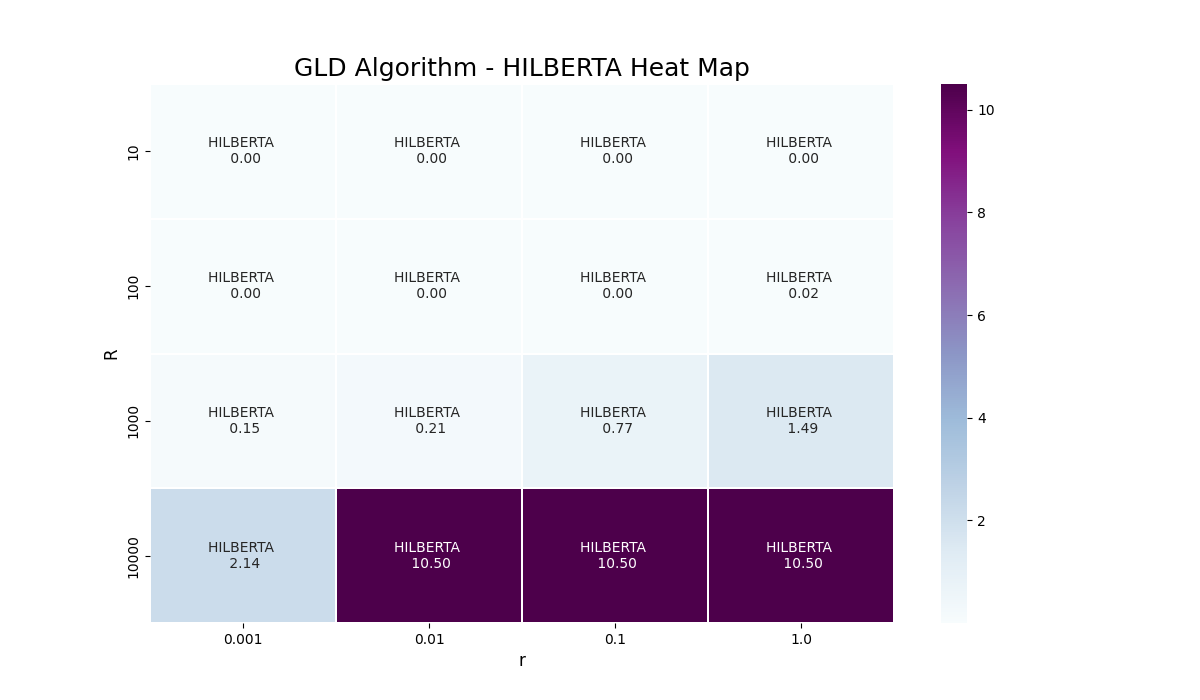}
  \hfill
  \includegraphics[width=0.32\textwidth]{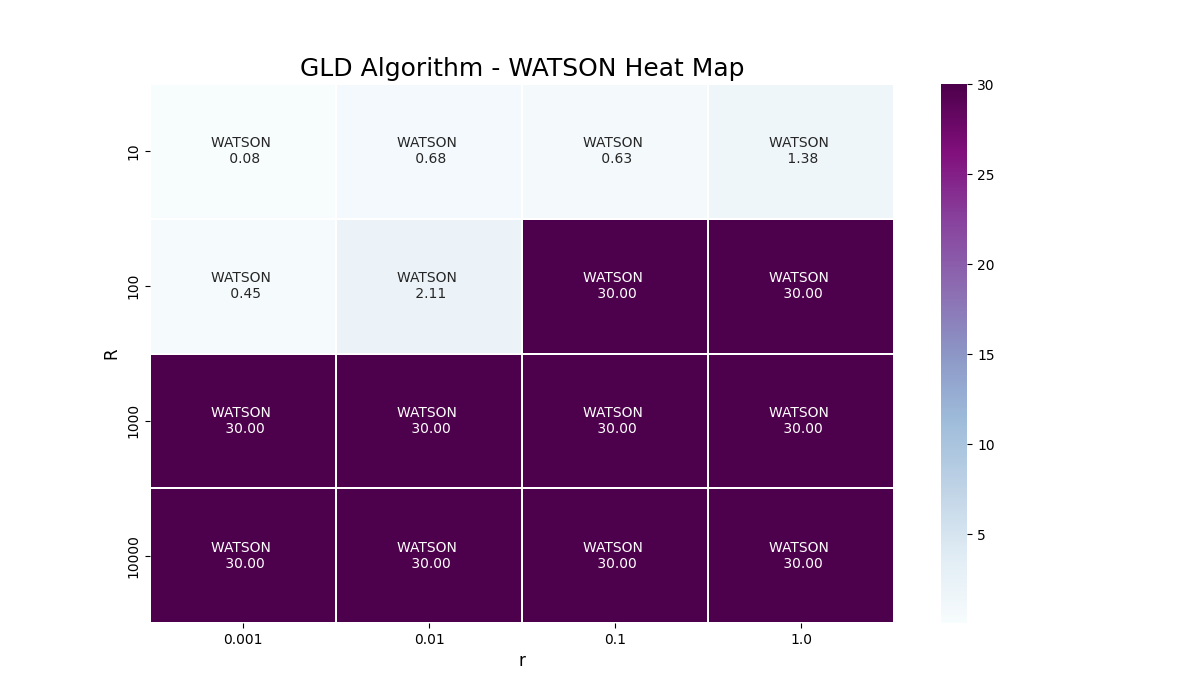}
  \caption{{\bf Top Left:} ROSENBR (100 queries). {\bf Top Center:} HILBERTA (100 queries). {\bf Top Right:} WATSON (100 queries). {\bf Bottom Left:} ROSENBR (5,000 queries). {\bf Bottom Center:} HILBERTA (5,000 queries). {\bf Bottom Right:} WATSON (5,000 queries).}
\end{figure}

{\tt GLD} takes in two scalar parameters $R$ and $r$, representing the upper and lower bounds of the search radii, respectively. \Cref{fig:heatmap_GLD} shows results in the form of heat map visualizations of {\tt GLD}’s performance for 100 and 5000 oracle queries. Darker colors represent higher final function evaluations and lighter colors represent lower final function evaluations. Thus, pairings of $r$, $R$ that are a lighter blue on the heat map represent more optimal parameter value pairings than those that are purple. We considered $r$ values of 0.001, 0.01, 0.1, 1.0, displayed on the $x$-axis, and $R$ values of 10000, 1000, 100, 10, displayed on the $y$-axis. \\

{\tt SCOBO} takes in scalar parameters $r$, $s$, and $m$, see \cite{cai2020one} for a discussion on the meaning of these parameters. Heat maps were generated by varying values of $s$ and $r$, while keeping $m$ fixed, and finding the last function evaluation for each of these pairings. The same was done for $s$ and $m$, keeping $r$ fixed. Results showing the heat maps for $s$ and $r$ variations are displayed in \Cref{fig:Heatmap_SCOBO_sr}, and heat maps for $s$ and $m$ variations are shown in \Cref{fig:Heatmap_SCOBO_sm}, containing graphs for both 1,000 and 10,000 oracle queries. Darker colors represent lower final function evaluations and greener colors represent higher final function evaluations (notice that this is a different color scale to \Cref{fig:heatmap_GLD}). Thus, pairings that are dark are better than those that are light. Values of $r$ range on the $x$-axis (0.001, 0.01, 0.1, 1.0) and values of $s$ range on the $y$-axis (100, 50, 20, 10).

\begin{figure}
  \centering
  \label{fig:Heatmap_SCOBO_sr}
  \includegraphics[width=0.32\textwidth]{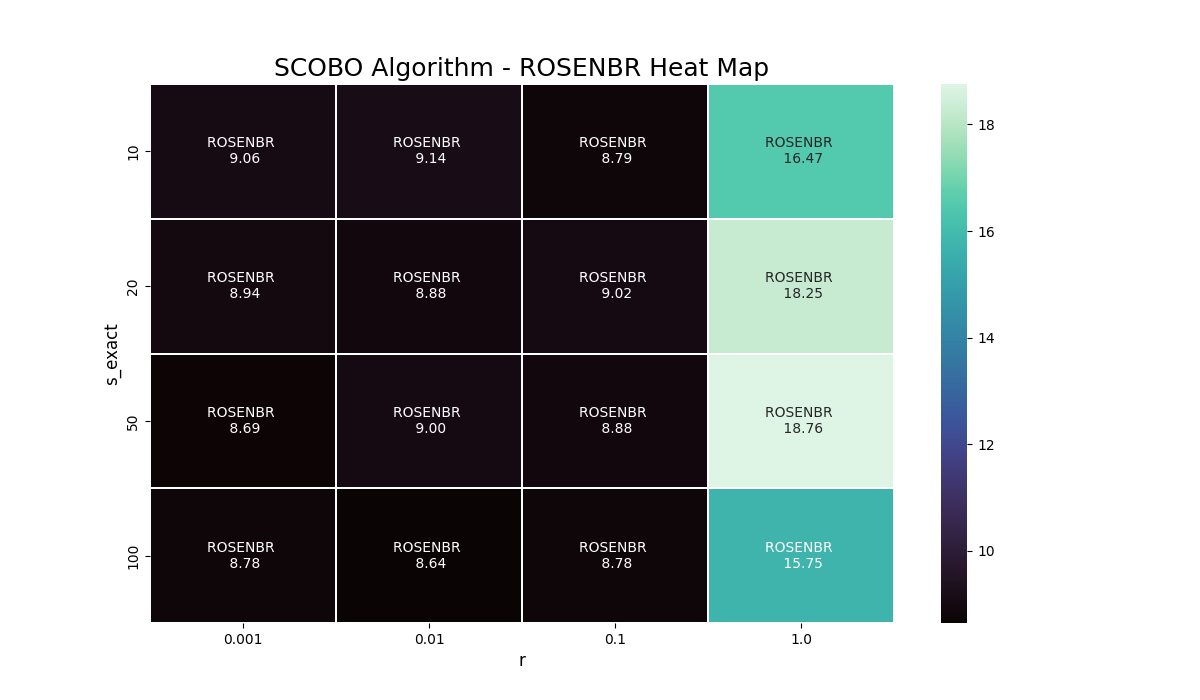}
  \hfill
  \includegraphics[width=0.32\textwidth]{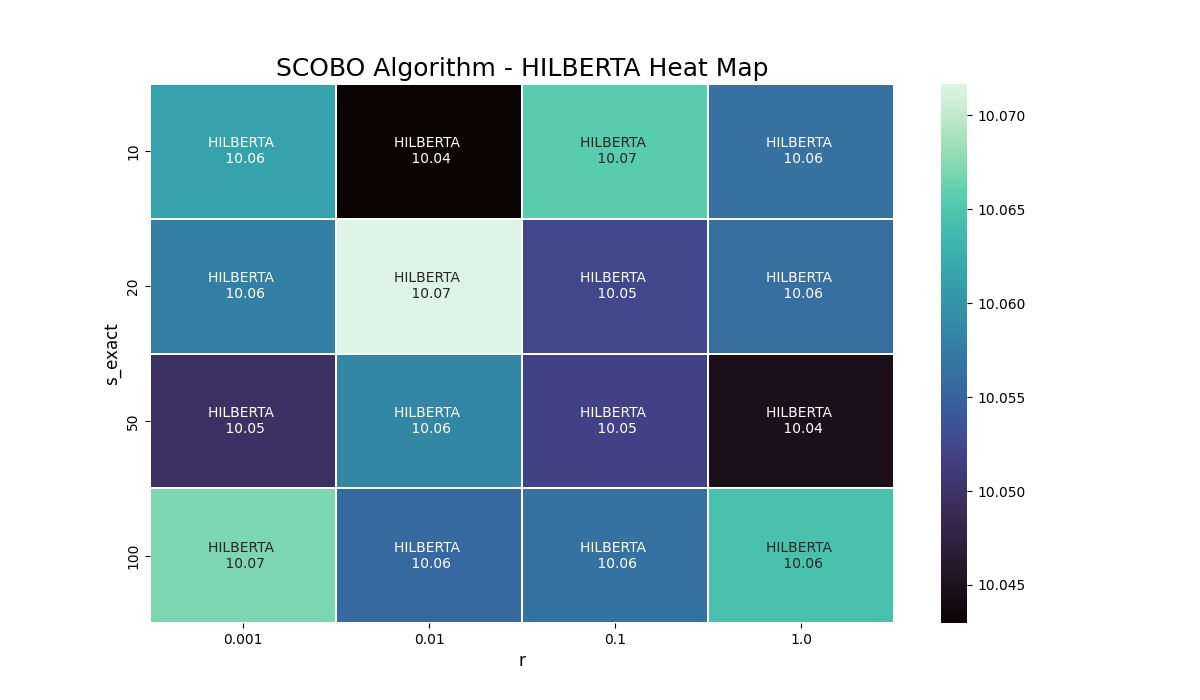}
  \hfill
  \includegraphics[width=0.32\textwidth]{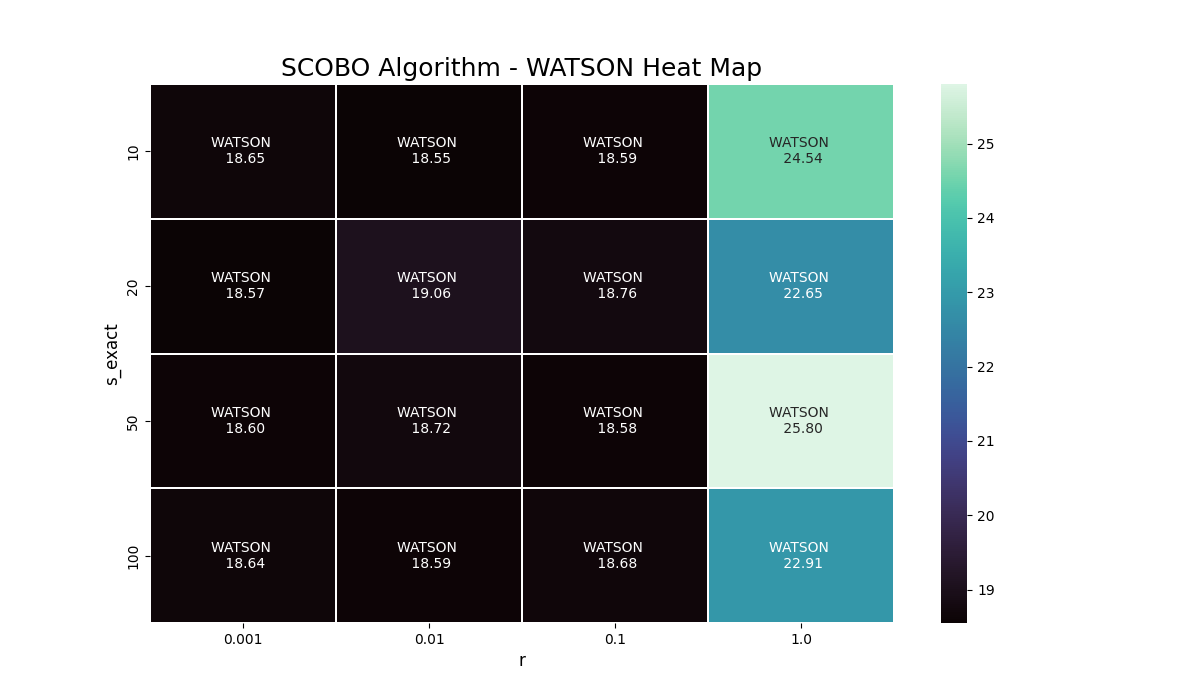}
  \hfill
  \includegraphics[width=0.32\textwidth]{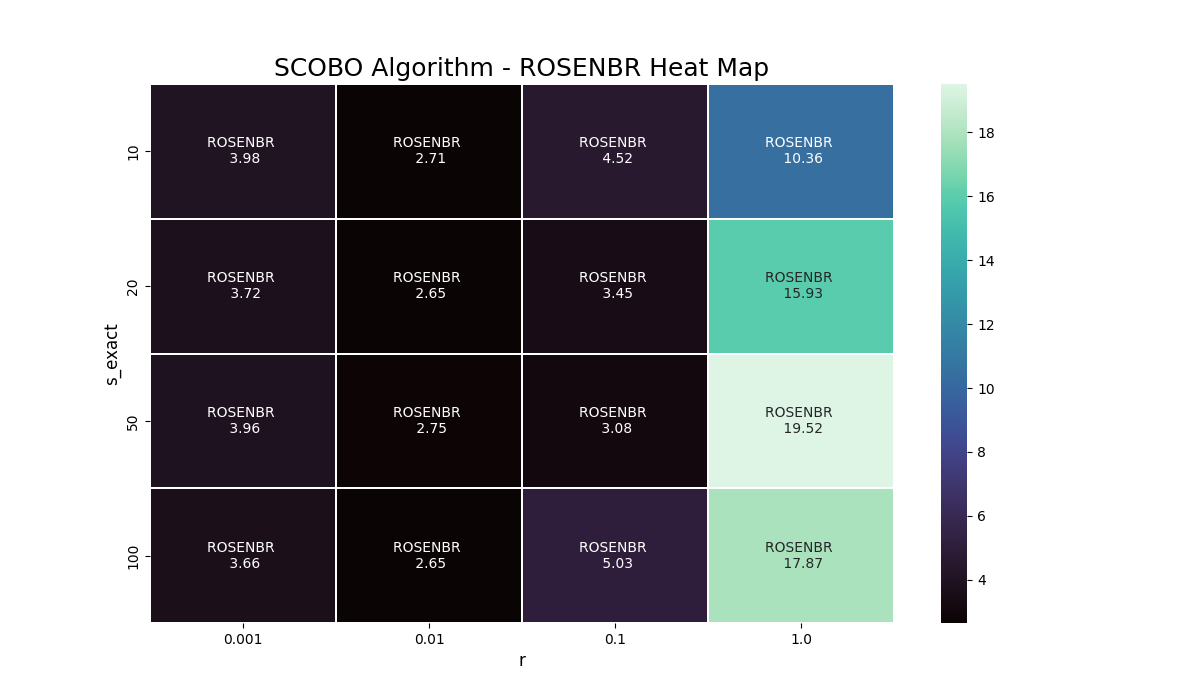}
  \hfill
  \includegraphics[width=0.32\textwidth]{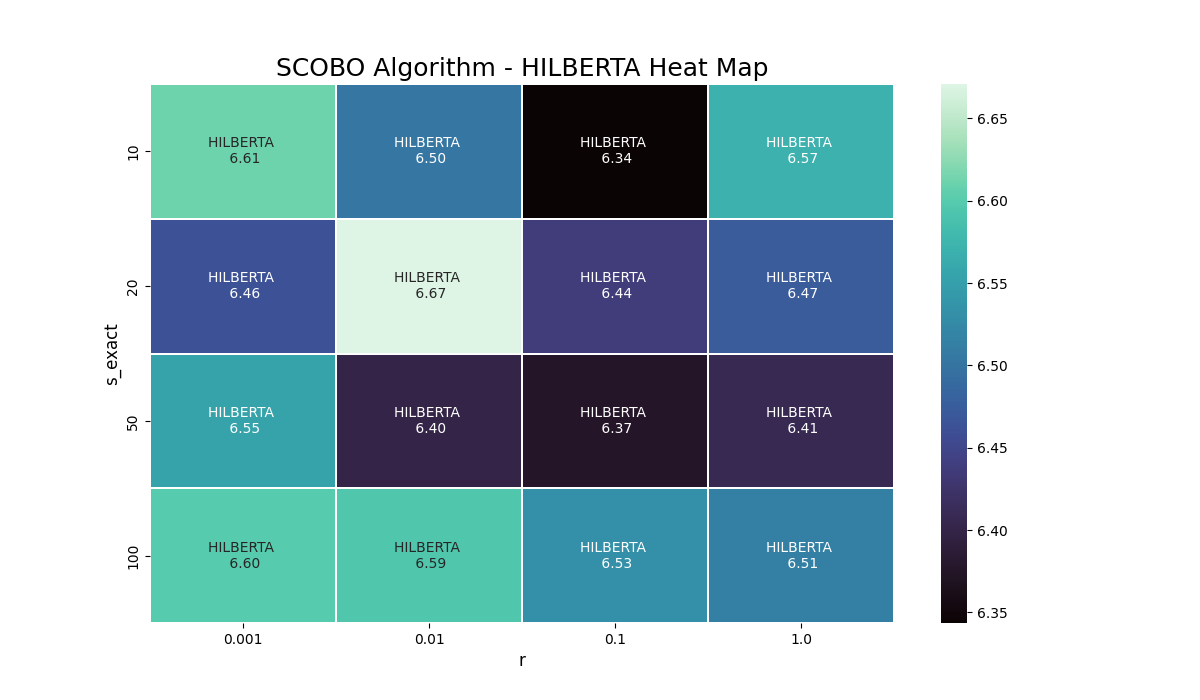}
  \hfill
  \includegraphics[width=0.32\textwidth]{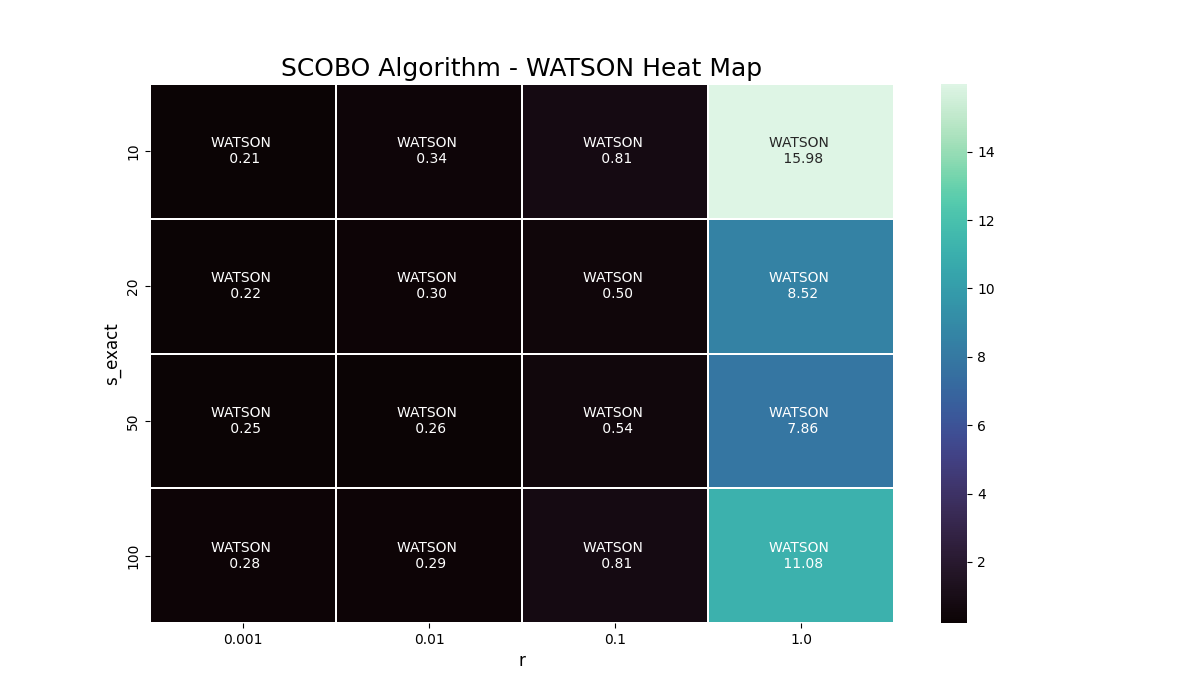}
  \caption{{\bf Top Left:} ROSENBR (1,000 queries). {\bf Top Center:} HILBERTA (1,000 queries). {\bf Top Right:} WATSON (1,000 queries). {\bf Bottom Left:} ROSENBR (10,000 queries). {\bf Bottom Center:} HILBERTA (10,000 queries). {\bf Bottom Right:} WATSON (10,000 queries).}
\end{figure}

\begin{figure}
  \centering
  \label{fig:Heatmap_SCOBO_sm}
  \includegraphics[width=0.32\textwidth]{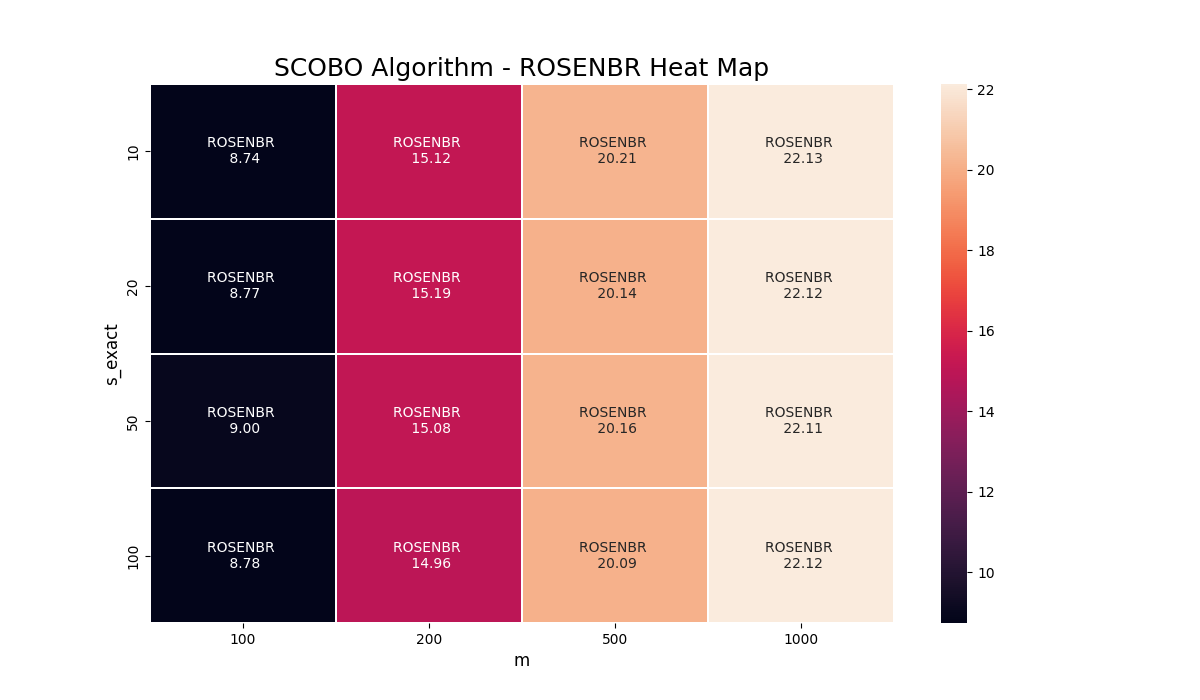}
  \hfill
  \includegraphics[width=0.32\textwidth]{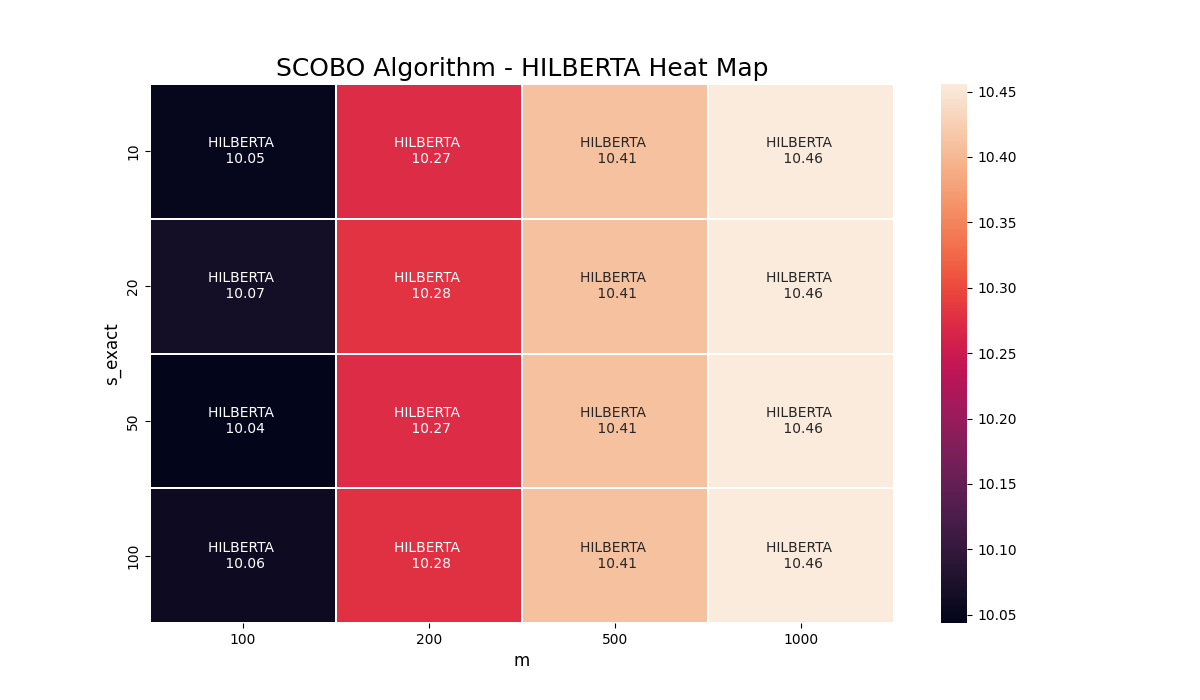}
  \hfill
  \includegraphics[width=0.32\textwidth]{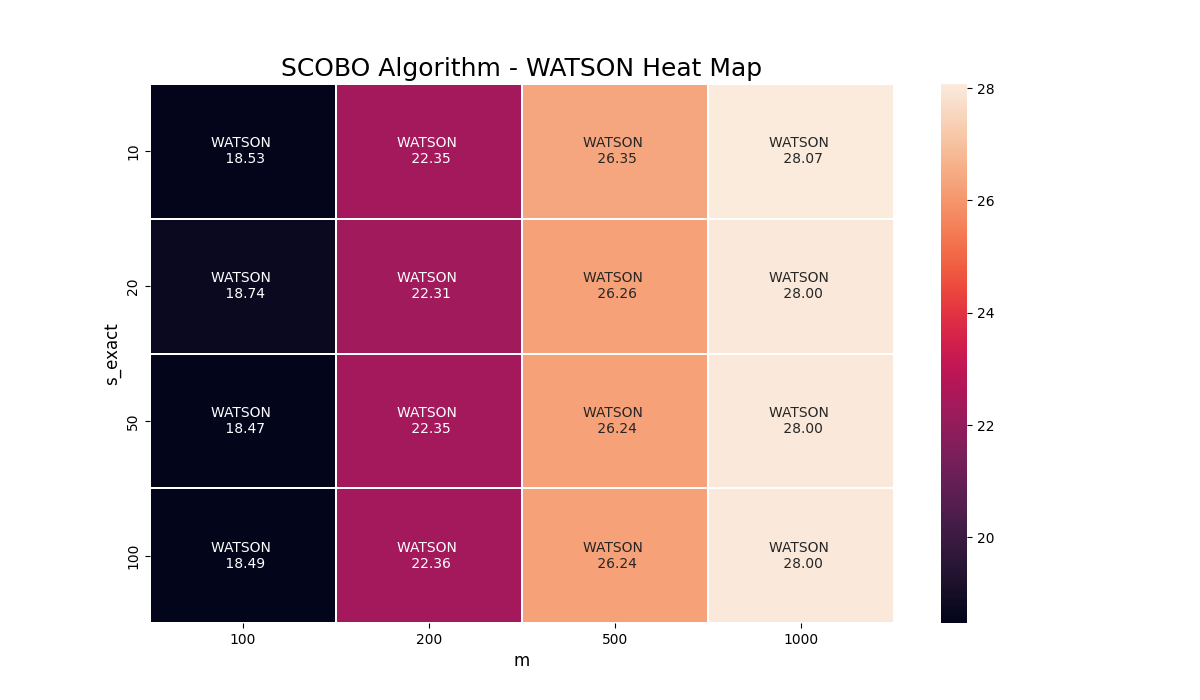}
  \hfill
  \includegraphics[width=0.32\textwidth]{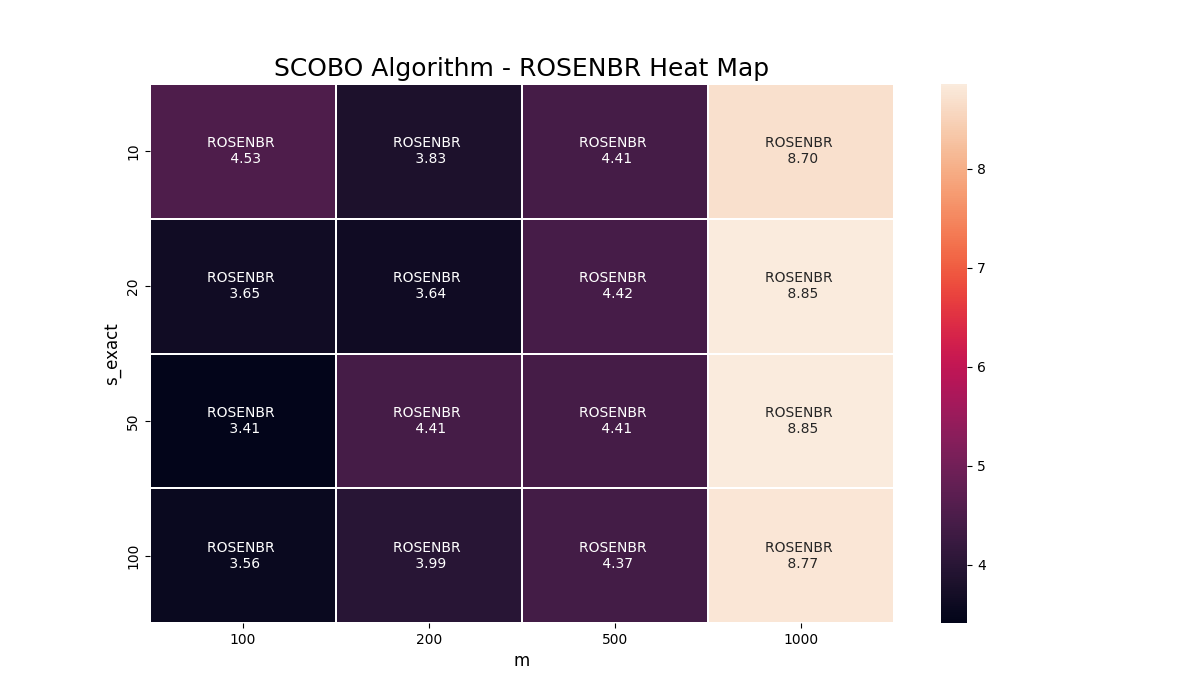}
  \hfill
  \includegraphics[width=0.32\textwidth]{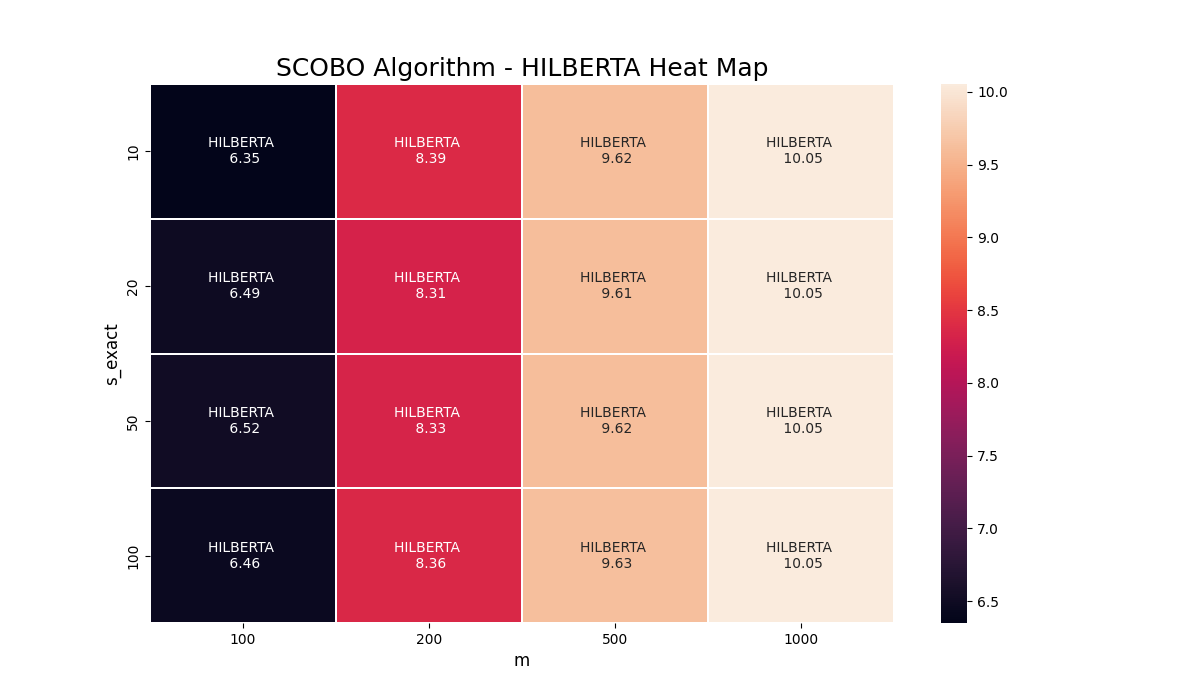}
  \hfill
  \includegraphics[width=0.32\textwidth]{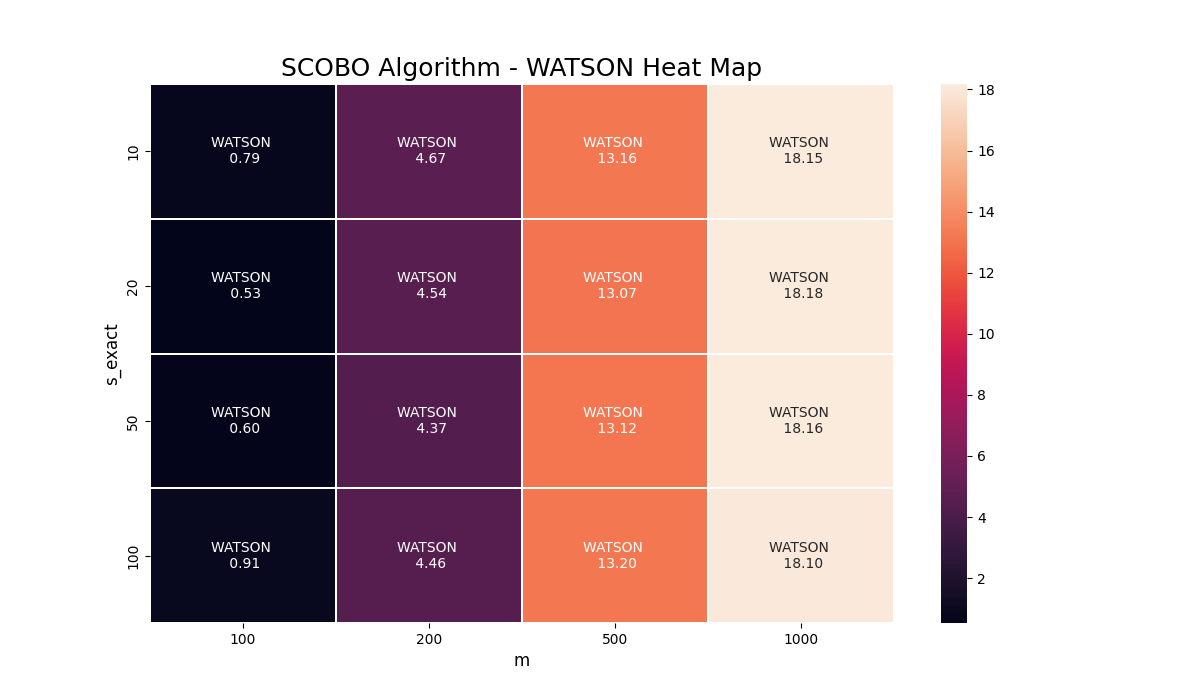}
  \caption{{\bf Top Left:} ROSENBR (1,000 queries). {\bf Top Center:} HILBERTA (1,000 queries). {\bf Top Right:} WATSON (1,000 queries). {\bf Bottom Left:} ROSENBR (10,000 queries). {\bf Bottom Center:} HILBERTA (10,000 queries). {\bf Bottom Right:} WATSON (10,000 queries).}
  \label{fig:heatmap_SCOBO}
\end{figure}

\section{Discussion}
\label{sec:discussion}

There are many insights to gather from the results of our experiments. Firstly, certain CBO algorithms work better than others in various situations. For example, notice that {\tt SignOPT} and {\tt CMA} tend to consistently lack in ability to optimize the SparseQuadratic, MaxK, and NonSparseQuadratic functions whereas {\tt GLD} and {\tt SCOBO} optimize them fast and efficiently (\Cref{fig:ToyProblemBenchmarking}). However in NonSparseQuadratic, a function without sparse gradients (which {\tt SCOBO} specializes in), GLD converges the function values faster. Therefore in this case it may be in a scientist's best interest to use {\tt GLD}. Furthermore we see in \Cref{fig:PyCUTEst_Results} that {\tt SCOBO} is unable to generalize to problems of higher dimensions whereas {\tt GLD} and {\tt STP} optimize the functions the best, relatively. Thus in the situation of non-sparse gradients and high dimensional input spaces, it may be in the scientist's favor to choose {\tt GLD}, {\tt STP}, or {\tt SignOPT} over {\tt SCOBO}. \\


Performance profiles provide information on the fraction of problems for which certain algorithms perform the best and on the robustness of each algorithm. Robustness refers to the ability of an algorithm to eventually solve hard problem instances. From the performance profiles shown in \Cref{fig:Performance_Profiles}, notice that each algorithm solves the same fraction of problems given the query budget. When success criterion is set to be $f(x_k)\leq 0.05f(x_0)$, we see approximately a $55\%$ solve rate given a $10^4$ budget and $65\%$ given a $10^5$ budget; when success criterion is defined as $\|\nabla f(x_k)\|_2 \leq 0.05\|\nabla f(x_0)\|_2$, we observe the same solve rates respective to query budgets apart from {\tt GLD} which is able to solve $80\%$ of problems given a $10^5$ query budget. Thus, while the robustness of each algorithm is relatively equal, this value is still low. However, a low value of robustness is expected as CBO is much harder than ZOO. Notice that {\tt STP} and {\tt SignOPT} have the highest rate of initial increase, indicating 
they are able to solve easier problems the fastest. {\tt SCOBO} clearly struggles the most in optimizing CUTEst functions. {\tt GLD} has a slower rate of initial increase yet eventually levels off at a high rate, indicating that if left with a high-enough query budget it can successfully solve a substantial portion of problems. \\

Focusing on the {\tt GLD} hyperparameter tuning experiments, we can determine that the best values of $R$ and $r$ tested were the smallest combination of each $(r = 0.001, R = 10)$. Observe that the top-left of each graph becomes the lightest in the smallest number of queries for each of the three PyCUTEst problems we tuned against. This corner of the heat map corresponds to the smallest value of each search radii bound (\Cref{fig:heatmap_GLD}). \\

For {\tt SCOBO}, we see interesting and perhaps less interpretable results. We first tuned hyperparameters $s$ and $m$ to find the best pairing. We found that $s$ had no effect on the performance of the pairing, whereas smaller $m$ led to better optimization. Recalling that $m$ denotes the number of queries made per iteration (see \cite{cai2020one}), this reveals an interesting tradeoff. Although higher $m$ means more accurate gradient approximations, it is better to take a smaller $m$ and hence less accurate gradient approximations, as this allows for more (albeit noisier) iterations given a fixed query budget. We next tuned $s$ and $r$ with $m=100$ fixed. Altering values for these parameters seemed to make no difference on performance when tuned against the problem HILBERTA, yet for ROSENBR and WATSON any pairing which had an $r$ value of $0.001$, $0.01$, or $0.1$ minimized the function the best. When $r = 1$, any value of $s$ yielded poor performance. For the ROSENBR problem, when $r = 0.01$ any value of $s$ had a very strong minimization. For the WATSON problem, $r = 0.01$ and $r = 0.001$ were both ideal. This indicates the best optimization occurs with a small $r$ value (\Cref{fig:Heatmap_SCOBO_sr}). \\

Finally, our experiments with a noisy oracle reveal some interesting insights into the robustness of the CBO algorithms considered. When the success rate is $p=0.9$, we hardly see a difference compared to using a normal Comparison Oracle (\Cref{fig:Noisy_Oracle_Results}). However, when the success rate is $p=0.7$ the benchmarked algorithms are noticeably more noisy. {\tt GLD}’s optimization is considerably changed as it is sensitive to error. This can be explained by the pseudocode \Cref{alg:stp original_2} and \Cref{alg:gld original}, as this algorithm takes strides in the direction of the smaller input value. Thus if the wrong value is being outputted by the oracle $30\%$ of the time, the algorithm will step in the wrong direction often and considerably diminish its minimization strength. While {\tt SCOBO}, {\tt SignOPT}, and {\tt STP} also seem to get noisier, they are much less impacted by the error. These three algorithms are more {\em robust} to noise than {\tt GLD} or {\tt CMA}. 

\section{Conclusions}
\label{sec:conclusions}

We have provided a novel utility for converting ZOO algorithms to CBO algorithms, and showcased how we did so with three state-of-the-art algorithms. We described our benchmarking experiments across these three converted and two already-existing CBO algorithms and analyzed results extensively. Users now have access to a suite of CBO algorithms as well as guidance in their application to continuous, large-scale CBO problems. \\ 

There is future work to conduct in this area. One idea is to use human comparison rather than a comparison oracle, so that instead of modeling what a noisy oracle may look like we establish it with human error. This continuation would tie into Cognitive Science, as we would work with human participants. Additionally, we can conduct hyperparameter tuning for more CBO algorithms, as we only covered two ({\tt GLD}, {\tt SCOBO}).

\section*{Acknowledgments}
Most of this work was conducted while the authors were affiliated with the Mathematics Department at UCLA. The authors are grateful for the stimulating environment provided by this department.

\bibliographystyle{siamplain}
\bibliography{references}

\newpage
\appendix
\section{Pseudocode}
\label{sec:main appendix}

This Appendix contains pseudocode for three of the five CBO algorithms ({\em STP, GLD CMA}). The other two algorithms, {\em SCOBO} and {\em SignOPT}, were originally CBO and not altered; thus, we did not feel the need to provide pseudocode for them as they can be found in their original respective papers (see references). The Appendix also contains pseudocode for three test problems ({\em Sparse Quadratic, MaxK, Non-Sparse Quadratic}), and four distributions ({\em Uniform distribution of canonical basis vectors, Gaussian, Uniform under Sphere, Rademacher}).


\begin{algorithm}
\caption{Stochastic Three Point ({\tt STP}). For the original algorithm use 3{\color{purple}a}. For the comparison-based version, use 3{\color{purple}b}.}
\label{alg:stp original_3}
\begin{algorithmic}
\STATE{\textbf{Initialization}}
\STATE{Choose $x_0 \in \mathbb{R}^{n}$, stepsizes $\alpha_k > 0$, probability distribution $\mathcal{D}$ on $\mathbb{R}^{n}$}
\FOR{$k = 0,1,2,....$}
  \STATE{1. Generate a random vector $s_k \sim \mathcal{D}$}
  \STATE{2. Let $x_+ = x_k + \alpha_k s_k$ and $x_- = x_k - \alpha_k s_k$}
  \STATE{3{\color{purple}a.} $x_{k+1} = \argmin\{f(x_-), f(x_+), f(x_k)\}$}
  \STATE{3{\color{purple}b.} $x_{k+1} = {\tt CompMin}(x_{-}, x_{+}, x_k)$}
\ENDFOR
\end{algorithmic}
\end{algorithm}

\begin{algorithm}
\caption{Gradientless Descent with Binary Search ({\tt GLD}). For the original algorithm use {\color{purple}a}. For the comparison-based version, use {\color{purple}b}.}
\label{alg:gld original}
\begin{algorithmic}
\STATE{\textbf{Initialization}}
\STATE{Take in function $f \colon \mathbb{R}^{n} \to \mathbb{R}$, $\mathcal{T} \in \mathbb{Z_+} \colon$ number of iterations, $x_0 \colon$ starting point, $\mathcal{D} \colon$ sampling distribution, $\mathcal{R} \colon$ maximum search radius, $r \colon$ minimum search radius}
\STATE{$\mathcal{K} = log(\mathcal{R}/r)$}
\FOR{$t = 0,...,\mathcal{T}$}
  \STATE{\textbf{Ball Sampling Trial:}}
  \FOR{$k = 0,...,\mathcal{K}$}
    \STATE{Set $r_k = 2^{-k}\mathcal{R}$}
    \STATE{Sample $v_k \sim r_{k}\mathcal{D}$}
  \ENDFOR
  \STATE{{\color{purple}a.} \textbf{Update:} $x_{t+1} = \argmin_{k}\{f(y) \vert y = x_{t}, y = x_{t} + v_{k}\}$}
  \STATE{{\color{purple}b.} \textbf{Update:} $x_{t+1} = {\tt CompMin}(x_{t}, x_{t} + v_{k})$}
\ENDFOR
\RETURN $x_{t}$
\end{algorithmic}
\end{algorithm}

\begin{algorithm}
\caption{Covariance Matrix Adaptation - Evolution Strategy ({\tt CMA-ES}). For the original algorithm use {\color{purple}a}. For the comparison-based version, use {\color{purple}b}.}
\label{alg:cma}
\begin{algorithmic}
\STATE{\textbf{Set parameters}}
\STATE{Set parameters $\lambda$, $w_{i=1\cdots \lambda}$, $c_{\sigma}$, $d_{\sigma}$, $c_c$, $c_1$, and $c_{\mu}$}
\STATE{\textbf{Initialization}}
\STATE{Set evolution paths $p_{\sigma} = 0$, $p_{c} = 0$, covariance matrix $\mathcal{C} = I$, and $g = 0$}
\FOR{$k=1,\ldots, \mathcal{K}$}
\STATE{Sample new population of search points, for $k = 1,\cdots,\lambda$}
\begin{eqnarray}
    z_{k} \sim \mathcal{N}(0,\mathcal{I}) \\
    y_{k} = \mathcal{B}\mathcal{D}z_{k} \sim \mathcal{N}(0,\mathcal{C}) \\
    x_{k} = m + \sigma y_{k} \sim \mathcal{N}(m,\sigma^{2}\mathcal{C})
\end{eqnarray}
\STATE{{\color{purple}a.} Sort:} Find permutation $\pi$ such that $f(x_{\pi(1)}) \leq f(x_{\pi(2)}) \leq \ldots \leq f(x_{\pi(\lambda)})$
\STATE{{\color{purple}b.} Sort:} $[x_{\pi(1)},x_{\pi(2)}, \ldots, x_{\pi(\lambda)}] = \text{\tt CompSort}(x_1,\ldots, x_n)$
\STATE{Recombination}
\begin{eqnarray}
    \langle y \rangle_{w} = \sum_{i=1}^{\mu} w_{i} y_{\pi(i)} \\ 
    m \leftarrow m + c_m \sigma \langle y \rangle_{w} 
\end{eqnarray}
\STATE{Step - size control}
\begin{eqnarray}
    p_{\sigma} \leftarrow (1-c_{\sigma})p_{\sigma} + \sqrt{c_{\sigma}(2-c_{\sigma})\mu_{eff}}\mathcal{C}^{-1/2}\langle y \rangle_{w} \\
    \sigma \leftarrow \sigma \times exp(\frac{c_{\sigma}}{d_{\sigma}}(\frac{||p_{\sigma}||}{\mathcal{E}||\mathcal{N}(0,\mathcal{I})||}-1))
\end{eqnarray}
\STATE{Covariance matrix adaptation}
\begin{eqnarray}
    p_{c} \leftarrow (1-c_c)p_c + h_{\sigma} \sqrt{c_{c}(2-c_{c})\mu_{eff}} \langle y \rangle_w \\ 
    w_{i}^{\circ} \leftarrow w_i \times (1\textrm{ if } w_i \geq 0 \textrm{ else } n/||C^{-\frac{1}{2}}y_{i \colon\lambda} ||^{2}) \\
    \mathcal{C} \leftarrow (1+c_1\delta (h_{\sigma})-c_1-c_{\mu}\sum w_j) \ \mathcal{C} + c_1 p_c p_c^{T} + c_{\mu} \sum_{i=1}^{\lambda} w_{i}^{\circ} y_{i \colon \lambda} y^{T}_{i \colon \lambda}
\end{eqnarray}
\ENDFOR
\end{algorithmic}
\end{algorithm}


\begin{algorithm}
\caption{Random Sampling Directions. \newline Based on type of probability distribution. {\color{purple}a}: Original (uniform distribution of canonical basis vectors). {\color{purple}b}: Gaussian distribution. {\color{purple}c}: Uniform under Sphere. {\color{purple}d}: Rademacher distribution. }
\label{alg:random sampling}
\begin{algorithmic}
\STATE{\textbf{Initialization}}

\STATE{Take in the following inputs: $x=$ number of direction vectors ($\emph{type=int}$), $y=$ length of each direction vector ($\emph{type=int}$), $z=$ the type of distribution ($\emph{original, gaussian, uniform under sphere, or rademacher}$)}
\IF{{\color{purple}a.} $z$ = "Original (uniform distribution of canonical basis vectors)"}
    \IF{$x = 1$}
        \STATE{randDirection $= w \in (0, y-1)$}
        \STATE{$s_k = [0, 0, ...., 0]$ where $len(s_k)=y$}
        \STATE{$s_k[$randDirection$]=1$}
    \ELSIF{$x>1$}
        \STATE{directionVectors $= [$ $]$}
        \FOR{$i$ from $0, \dots, x-1$}
            \STATE{randDirection $= w \in (0, y-1)$}
            \STATE{$s_k = [0, 0, \dots, 0]$ where $len(s_k)=y$}
            \STATE{$s_k[$randDirection$]=1$}
            \STATE{directionVectors.append($s_k$)}
        \ENDFOR
        \STATE{$s_k \in \mathbb{R}^{x*y}$ = [ [$v_1$], [$v_2$], \dots, [$v_x$] ] for $v_i \in$ directionVectors, $i \in \{1 \dots x\}$, $v_i \in \mathbb{R}^{y}$}
    \ENDIF
\ENDIF
\IF{{\color{purple}b.} $z$ = "Gaussian Distribution"}
    \IF{$ x = 1$}
        \STATE{$s_k = [d_1, d_2, ..., d_y]$ where $\{d_i\} \in$ standard normal distribution; $i \in \{1, \dots ,y\}$}
    \ELSIF{$x > 1$}
        \STATE{$s_k \in \mathbb{R}^{x*y}$ = [ [$v_1$], [$v_2$], \dots, [$v_x$] ] for $v_i \in$ standard normal distribution $\subset \mathbb{R}^{y}$, $i \in \{1 \dots x\}$}
    \ENDIF
\ENDIF
\STATE{\emph{Continued on next page....}}
\end{algorithmic}
\end{algorithm}

\begin{algorithm}
\caption{Random Sampling Directions \textbf{Cont'd.}}
\label{alg:random sampling 2}
\begin{algorithmic}
\STATE{\dots .}
\IF{{\color{purple}c.} $z$ = "Uniform Under Sphere"}
    \IF{$x = 1$}
        \STATE{$s_k = [d_1, d_2, ..., d_y]$ where $\{d_i\} \in$ standard normal distribution; $i \in \{1, \dots ,y\}$}
        \STATE{norm $=$ Frobenius norm of $s_k$}
        \STATE{$s_k = s_k$ $/$ norm}
    \ELSIF{$x > 1$}
        \STATE{directionVectors $= [$ $]$}
        \FOR{$i$ from $0, \dots, x-1$}
            \STATE{$s_k = [d_1, d_2, ..., d_y]$ where $\{d_i\} \in$ standard normal distribution; $i \in \{1, \dots ,y\}$}
            \STATE{norm $=$ Frobenius norm of $s_k$}
            \STATE{$s_k = s_k$ $/$ norm}
            \STATE{directionVectors.append($s_k$)}
        \ENDFOR
        \STATE{$s_k \in \mathbb{R}^{x*y}$ = [ [$v_1$], [$v_2$], \dots, [$v_x$] ] for $v_i \in$ directionVectors, $i \in \{1 \dots x\}$, $v_i \in \mathbb{R}^{y}$}
    \ENDIF
\ENDIF
\IF{{\color{purple}d.} $z$ = "Rademacher Distribution"}
    \IF{$x = 1$}
        \STATE{$s_k = 2$ $*$ round$([d_1, d_2, ..., d_y]) + 1$ where $\{d_i\} \in$ uniform distribution over [0, 1); $i \in \{1, \dots ,y\}$}
        \STATE{$s_k = s_k / \sqrt{y}$}
    \ELSIF{$x > 1$}
        \STATE{directionVectors $= [$ $]$}
        \FOR{$i$ from $0, \dots, x-1$}
            \STATE{$s_k = 2$ $*$ round$([d_1, d_2, ..., d_y]) + 1$ where $\{d_i\} \in$ uniform distribution over [0, 1); $i \in \{1, \dots ,y\}$}
            \STATE{$s_k = s_k / \sqrt{y}$}
            \STATE{directionVectors.append($s_k$)}
        \ENDFOR
        \STATE{$s_k \in \mathbb{R}^{x*y}$ = [ [$v_1$], [$v_2$], \dots, [$v_x$] ] for $v_i \in$ directionVectors, $i \in \{1 \dots x\}$, $v_i \in \mathbb{R}^{y}$}
    \ENDIF
\ENDIF
\RETURN{$s_k$}
\end{algorithmic}
\end{algorithm}

\end{document}